\documentclass[reqno]{amsart}

\usepackage{mathrsfs} 
\usepackage{amsmath} 
\usepackage{amssymb} 
\usepackage{cite} 
\usepackage{latexsym} 
\usepackage{graphicx} 

\usepackage{color} 

\theoremstyle{plain} 
\newtheorem{thm}{Theorem}[section] 
\newtheorem{lem}[thm]{Lemma} 
\newtheorem{cor}[thm]{Corollary} 
 
\newtheorem{prop}[thm]{Proposition} 
\newtheorem{rmk}[thm]{Remark}

\def\D{\mathrm{D}}

\def\T{\mathrm{T}} 

\def\c{\mathrm{c}} 
\def\d{\mathrm{d}}

\def\Nset{\mathbb{N}} 
\def\Pset{\mathbb{P}} 
\def\Rset{\mathbb{R}} 
\def\Tset{\mathbb{T}} 
\def\Sset{\mathbb{S}} 
\def\Zset{\mathbb{Z}}

\def\epsilon{\varepsilon}

\numberwithin{equation}{section}

\begin{document} 


\title[Feedback control of the Kuramoto model I]%
{Feedback control of the Kuramoto model defined on uniform graphs I:
 Deterministic natural frequencies}
%
%
\author{Kazuyuki Yagasaki} 
\address{Department of Applied Mathematics and Physics, 
Graduate School of Informatics, Kyoto University, 
Yoshida-Honmachi, Sakyo-ku, Kyoto 606-8501, JAPAN} 

\email{yagasaki@amp.i.kyoto-u.ac.jp} 

\date{\today} 

\begin{abstract}
We consider feedback control of the Kuramoto model on uniform graphs,
where the natural frequencies are uniformly spaced and the graphs may be
complete, random dense or random sparse.
The control objective is to drive all nodes to the same constant rotational motion.
For the case of node number $n\ge 3$,
 we establish the existence of exactly $2^n$ synchronized solutions in the controlled Kuramoto model (CKM)
 and their saddle-node and pitchfork bifurcations, and determine their stability.
In particular,  we show that
 only a solution converging to the desired motion in the limit of infinite feedback gain is stable 
 and the others are unstable.
Based on the previous results,
 it is shown that (i) the solution to which the stable synchronized solution in the CKM converges as $n\to\infty$
 is always asymptotically stable in the continuous limit (CL) if it exists,
 and (ii) the asymptotically stable solution of the CL
 captures the asymptotic behavior of the CKM when the node number is sufficiently large,
 even if the graphs are random dense or sparse.
We give numerical simulations of the CKM
 on complete simple graphs and on uniform random dense and sparse graphs
 to demonstrate our theoretical results.

\end{abstract}

\maketitle 

\section{Introduction}

\subsection{Kuramoto model defined on graphs}

We consider the Kuramoto model (KM) \cite{K75,K84}
 on graphs $G_n$, $n\in\Nset$, with natural frequencies depending on the nodes,
\begin{align} 
\frac{\d}{\d t}u_i^n(t) = 
&\omega_i^n+ \frac{K}{n \alpha_n} \sum_{j=1}^{n} 
w_{ij}^n \sin\left(u_j^n(t)-u_i^n(t)\right),\quad
i\in[n]:=\{1,2,\ldots,n\},
\label{eqn:dsys0} 
\end{align} 
where $n$ is the node number;
 $u_i^n(t)\in\Sset^1=\Rset/2\pi\Zset$ and $\omega_i^n$,
 respectively, represent the phase of the $i$th oscillator and its natural frequency;
 $K>0$ controls the strength of the coupling;
 and $\alpha_n$ is a scaling factor chosen to be one  if $G_n$ is dense.
Here $G_n=\langle V(G_n),E(G_n),W(G_n)\rangle$, $n\in\Nset$, 
 represent a sequence of weighted graphs,
 where $V(G_n)=[n]$ and $E(G_n)$ are the sets of nodes
 and edges, respectively, and $W(G_n)$ is an $n\times n$ 
 weight matrix given by 
\begin{equation*} 
(W(G_n))_{ij}= 
\begin{cases} 
w_{ij}^n & \mbox{if $(i,j)\in E(G_n)$};\\ 
0 &\rm{otherwise.} 
\end{cases} 
\end{equation*} 
The edge set is expressed as
\begin{equation*}
E(G_n)=\{(i,j)\in[n]^2 \mid (W(G_n))_{ij} \neq 0\},
\end{equation*}
where the ordered pair $(i,j)$ denotes an edge from $j$ to $i$.
This edge is also written as $j\to i$ and loops are allowed.
When $W(G_n)$ is symmetric, 
 $G_n$ is regarded as an undirected weighted graph 
 and each edge is written as $i\sim j$ rather than $j\to i$. 
When $G_n$ is a simple graph (resp. a random graph),
 $W(G_n)$ is a $0$-$1$ matrix (resp. a random matrix).
We call the graph $G_n$ \emph{dense} 
 if $\lim_{n\to\infty}\# E(G_n)/(\# V(G_n))^2>0$, 
 and \emph{sparse}
 if $\lim_{n\to\infty}\# E(G_n)/(\# V(G_n))^2=0$. 
 
The weight matrix $W(G_n)$ is determined as follows:
Let $I=[0,1]$ and let $W\in L^2(I^2)$ be a measurable function.
If $G_n$, $n\in\Nset$, are deterministic dense graphs, then
\begin{equation}
w_{ij}^n = \langle W\rangle_{ij}^n
:= n^2 \int_{I_i^n \times I_j^n}W(x,y) \d x\d y,
\label{eqn:ddg}
\end{equation}
where
\[
I_i^n:=\begin{cases}
[(i-1)/n,i/n) & \mbox{for $i<n$};\\
[(n-1)/n,1] & \mbox{for $i=n$}.
\end{cases}
\]
If $G_n$, $n\in\Nset$, are random dense graphs, 
then $w_{ij}^n=1$ with probability 
\begin{equation}
\mathbb{P}(j \rightarrow i) = \langle W\rangle_{ij}^n, 
\label{eqn:rdg}
\end{equation}
where the range of $W$ is contained in $I$.
If $G_{n}$, $n\in\Nset$, are random sparse graphs, 
then $w_{ij}^n=1$ with probability 
\begin{equation} 
\mathbb{P}(j \rightarrow i) = \alpha_n 
\langle \tilde{W}_n \rangle_{ij}^n, \quad 
\tilde{W}(x,y) := \alpha_n^{-1} \wedge W(x,y), 
\label{eqn:rsg} 
\end{equation} 
where $W$ is a nonnegative function,
 $\alpha_n =n^{-\gamma}$ with $\gamma\in(0,\frac{1}{2})$,
 and $a\wedge b=\min(a,b)$ for $a,b\in\Rset$.
For the random cases,
 the graph $G_n$ is a \emph{Erd\"os–R\'enyi} model.
In both cases, the connection probabilities are determined
 by a measurable function $W(x,y)$ defined on $I^2$,
 which is usually called a \emph{graphon}\cite{L12}.

Such coupled oscillators have attracted considerable attention in recent years
 and have been investigated intensively.
Many mathematical models based on them
 have been proposed and used to describe collective phenomena
 in a wide range of fields, including natural sciences, social sciences and engineering.
In particular, the KM occupies a central position
 and has led to numerous generalizations
 such as models with different graph structures,
 coupling functions and natural frequencies.
Among the phenomena exhibited by these models,
 synchronization has attracted particular interest.
We refer to \cite{S00,PRK01,ABVRS05,ADKMZ08,DB14,PR15,RPJK16}
 for surveys on coupled oscillator networks,
 including the KM and its generalizations.
For the classical KM with uniformly spaced natural frequencies,
 all synchronized solutions and their stability were completely characterized in \cite{Y24a}.
Related stability results for KMs on general coupling graphs
 can be found, for example, in \cite{DJD19, HKR16}.

\subsection{Continuum limits and graphons}
In the previous work \cite{IY23},
 coupled oscillator networks such as \eqref{eqn:dsys0}
 were analyzed via their continuum limits (CLs).
For \eqref{eqn:dsys0}, the corresponding CL is given by
\begin{equation}
\frac{\partial}{\partial t} u(t,x)
=\omega(x)+K \int_I W(x,y) \sin(u(t,y)-u(t,x)) \d y,\quad x\in I,
\label{eqn:csys0}
\end{equation}
provided that the natural frequencies are generated from an $L^2$ function $\omega$ on $I$ by
\begin{equation}
\omega_i^n=n\int_{I_i^n}\omega(x)dx, \quad i \in [n].
\label{eqn:omegai0}
\end{equation}
We refer to $\omega(x)$ as a \emph{frequency function}
 and assume that it is not a constant function.
More general situations,
 where the oscillator networks depend on multiple graphs,
 each of which may be deterministic or random, and dense or sparse,
 were treated in \cite{IY23}.
In \cite{Y24a},
 some stability results for the coupled oscillator systems and their CLs were refined.
Earlier related results for single-graph networks with identical natural frequencies
 were given in \cite{KM17,M14a,M14b,M19},
 although they contain no stability results for solutions
 and do not apply to \eqref{eqn:dsys0} and \eqref{eqn:csys0}.

A CL was introduced in \cite{E85} without a rigorous mathematical justification,
 for the classical KM, which is defined on a single complete simple graph
 and has node-dependent natural frequencies.
The situation where the natural frequencies are equally placed,
 for arbitrary odd node number $n\ge 3$, was analyzed in detail
 much more recently in \cite{Y24a}.
In particular, it was shown there that
 the bifurcations and stability of synchronized solutions in the classical KM
 differ substantially from those in its CL.
Furthermore, bifurcations of completely synchronized and twisted solutions in CLs
 were studied in \cite{Y24b,Y24c,Y24d} for KMs of identical oscillators.
The models considered there have two-mode interactions depending on two graphs
 or are defined on nearest-neighbor graphs in the presence or absence of feedback control.
These studies utilized the center manifold reduction technique \cite{HI11},
 which is a standard tool in the theory of dynamical systems \cite{GH83,W03}.
Similar CLs were also used in \cite{GHM12,M14c,MW17,WSG06}
 for nonlocally coupled KMs with a single or zero natural frequency.

\subsection{Related work on feedback control of coupled oscillators}
The control problem of coupled oscillator networks is important
 not only from a theoretical viewpoint but also for applications,
 and has attracted much attention \cite{CKM13,DNM22,M15}.
Such studies are also significant in practice,
 since situations frequently arise
  in which all coupled oscillators are desired to exhibit the same rhythmic motion such as heart beats.
In particular, for Kuramoto-type models,
 various control strategies have been proposed depending on the network structure
 and the form of control inputs.

For instance, coupled oscillator networks defined on multiple graphs
 have been studied from the viewpoint of control \cite{GC20,GC21,IY23,Y24b}.
In these works, additional coupling graphs are introduced
 to avoid undesired synchronization,
 which may otherwise lead to traffic congestion or collapse of networked systems.
 In particular, in \cite{Y24b}, it was shown that synchronized solutions,
 which are intended to be avoided, can be unstable,
 and their stability and bifurcations were analyzed theoretically.

On the other hand, feedback control of synchronized states
 in the KM on deterministic dense, random dense and random sparse graphs
 has been studied numerically or theoretically in \cite{SA15,LR17,WL20,IY23,Y24d}.
In these studies, a prescribed desired motion or position is incorporated into the dynamics
 as an external input, which is often periodic.
Such a periodic input is referred to as a \emph{pacemaker} in \cite{LR17,WL20}.
However, the stability and bifurcations of the desired synchronized state
 were not clarified theoretically in \cite{SA15,LR17,WL20,IY23},
 although those of twisted solutions chosen as desired states were analyzed theoretically in \cite{Y24d}.

\subsection{Object of the paper}
In this paper, we deal with the case of uniform graphs, i.e., $W(x,y)=p$,
 which  may be deterministic or random, and dense or sparse.
We attempt to control the KM \eqref{eqn:dsys0}
 so that each node exhibits the desired motion $V(t)=V_1t+V_0$, 
 where $V_1,V_0$ are constants,
 although more general desired motion can be treated similarly.
So we add the control force $b_1 \sin(V(t)-u_i^n(t))+b_0$  to each node as
\begin{align} 
\frac{\d}{\d t}u_i^n(t)=&\omega_i^n+ \frac{K}{n \alpha_n} \sum_{j=1}^{n} 
w_{ij}^n \sin\left(u_j^n(t)-u_i^n(t)\right)\notag\\
&+b_1 \sin(V(t)-u_i^n(t))+b_0,\quad i\in[n],
\label{eqn:dsys} 
\end{align}
where the constants $b_1$ and $b_0$
 represent the feedback gain and stationary input, respectively.
Here we include the stationary input $b_0$ for simplicity of analysis.
The corresponding CL is given by
\begin{align}
\frac{\partial}{\partial t} u(t,x) = 
& \omega(x)+pK \int_{I}\sin(u(t,y)-u(t,x)) \d y\notag\\
&+b_1 \sin(V(t)-u(t,x))+b_0.
\label{eqn:csys}
\end{align}
A restricted case was studied
 and some numerical simulation results were provided in \cite{IY23}.
We remark that the CL \eqref{eqn:csys} is the same
 even if the graphs $G_n$, $n\in\Nset$, are complete simple, random dense or random sparse. 

Let 
\begin{equation}
b_0 = V_1- \int_I \omega(x) \d x.
\label{eqn:conb0}
\end{equation}
We see that the CL \eqref{eqn:csys} admits two particular solutions
\begin{equation}
u(t,x) = U(x)+V(t)\quad\mbox{and}\quad
u(t,x) = \pi-U(x)+V(t),
\label{eqn:csol}
\end{equation}
where
\begin{equation}
U(x) = \arcsin\left(\frac{\omega(x)-V_1+b_0}{pKC+b_1}\right),
\label{eqn:U}
\end{equation}
if there exist constants $C>0$ and $C<0$, respectively, such that
\begin{equation}
C = \pm\int_I \sqrt{1-\left(\frac{\omega(x)-V_1+b_0}{pKC+b_1}\right)^2} \d x,
\label{eqn:C}
\end{equation}
where the upper and lower signs, respectively, 
 correspond to  the cases $C>0$ and $C<0$.
Note that
\begin{equation}
\sup_{x\in I}\,|\omega(x)-V_1+b_0|\leq pKC+b_1.
\label{eqn:conc}
\end{equation}
In particular, the solution \eqref{eqn:csol} is continuous in $t$ and $x$ if $\omega(x)$ is continuous.
On the other hand, let
\begin{equation}
b_0 = V_1- \frac{1}{n} \sum_{i=1}^n \omega_i^n.
\label{eqn:b0}
\end{equation}
When the graphs $G_n$, $n\in\Nset$, are complete,
 i.e., $w_{ij}^n=p$, $i,j\in[n]$,
 the controlled Kuramoto model (CKM) \eqref{eqn:dsys} has a particular solution
\begin{equation}
u_i^n(t) = U_i^n+V(t)\quad\mbox{and}\quad
u_i^n(t) = \pi-U_i^n+V(t),
\label{eqn:dsol}
\end{equation}
where
\[
U_i^n = \arcsin\left(\frac{\omega_i^n-V_1+b_0}{pKC_\D+b_1}\right),
\]
if there exist constants $C_\D>0$ and $C_\D<0$, respectively, such that
\begin{equation}
C_\D = \pm\frac{1}{n} \sum_{i=1}^{n} \sqrt{1-\left(\frac{\omega_i^n-V_1+b_0}{pKC_\D+b_1}\right)^2},
\label{eqn:CD0}
\end{equation}
where the upper and lower signs, respectively,
 correspond to  the cases $C_\D>0$ and $C_\D<0$.
Note that
\begin{equation}
\max_{i\in[n]}\,|\omega_i^n-V_1+b_0|\leq pKC_\D+b_1.
\label{eqn:cond}
\end{equation}
See Appendix~A for the derivation of \eqref{eqn:csol} and \eqref{eqn:dsol}.
In particular, as $b_1\to\infty$, the solutions  \eqref{eqn:csol} and \eqref{eqn:dsol},
 respectively, tend to
\[
u(t,x)=V(t)\quad\mbox{and}\quad 
u_i^n(t) = V(t),
\]
which coincide with the desired motions completely.
We will show below that these are far from all synchronized solutions:
 the number of such solutions in the CKM \eqref{eqn:dsys} is $2^n$,
 whereas the CL \eqref{eqn:csys} has uncountably many synchronized solutions.

Specifically, we choose
\begin{equation}
\omega(x)=a(x-\tfrac{1}{2}),
\label{eqn:omega}
\end{equation}
where $a>0$ is a constant,
 as the frequency function as in \cite{Y24a},
 so that by \eqref{eqn:omegai0}
\begin{equation}
\omega_i^n=\frac{a}{2n}(2i-n-1),\quad
i\in[n],
\label{eqn:omegai}
\end{equation}
i.e., the natural frequencies are uniformly spaced.
In addition, by \eqref{eqn:conb0} and \eqref{eqn:b0} $b_0=V_1$.
Under the general assumption that the node number $n$ satisfies $n\ge 3$,
 we prove that for $a,pK>0$ fixed
 the solution \eqref{eqn:dsol} to the CKM \eqref{eqn:dsys}
 suffers a saddle-node bifurcation at some value of $b_1$
 where a pair of stable and unstable ones are born,
 and the solution \eqref{eqn:csol} is always asymptotically stable
 in the CL \eqref{eqn:csys} whenever it exists,
 as in the uncontrolled KM \eqref{eqn:dsys0} and its CL \eqref{eqn:csys0}
 (see \cite{Y24a}).
They are also proven to be the only synchronized solutions
 that are or can be stable in the CKM \eqref{eqn:dsys} and CL \eqref{eqn:csys}, respectively,
 when $b_1>0$ and $K>0$ is small enough for them not to exist for $b_1,b_0=0$.
Moreover, we show based on the results of \cite{IY23,Y24a}
 that the solution \eqref{eqn:csol} to the CL \eqref{eqn:csys} behaves
 as it is an asymptotically stable one in the CKM \eqref{eqn:dsys} for $n>0$ sufficiently large,
 even if the graphs are random dense or sparse.
In the companion paper \cite{KY24},
 we discuss a similar control problem for the KM \eqref{eqn:dsys0}
 with random natural frequencies.

\subsection{Outline of the paper}
The outline of this paper is as follows.
In Section~2, we review fundamental theoretical results
 from \cite{IY23} and \cite{Y24a} for coupled nonlinear oscillator networks
 in the context of \eqref{eqn:dsys} and \eqref{eqn:csys}.
These results play an important role in the analysis in Sections 4 and 5.
In Section~3, we present our main results
 on synchronized solutions of the CKM \eqref{eqn:dsys} on complete graphs.
The CKM \eqref{eqn:dsys} is transformed into an equivalent autonomous system,
 all its equilibria, the number of which is exactly $2^n$, are characterized,
 and their bifurcations and stability are determined.
In particular, we prove that the equilibrium corresponding to the solution \eqref{eqn:dsol}
 undergoes a saddle-node bifurcation, at which a pair of stable and unstable equilibria is created,
 while all other equilibria are unstable when $b_1>0$ and $K>0$ is sufficiently small
 so that no equilibrium exists for $b_1=b_0=0$.
In Section~4, we analyze the CL \eqref{eqn:csys}
 and show that it has uncountably many synchronized solutions.
Moreover, it is proved  that
 the solution $u=U(x)+V(t)$ in \eqref{eqn:csol} is always asymptotically stable whenever it exists,
 and all other continuous and discontinuous synchronized solutions are unstable
 when $b_1>0$ and $K>0$ is sufficiently small
 so that no synchronized solution exists for $b_1=b_0=0$.
In Section~5, we discuss the behavior of the CKM \eqref{eqn:dsys}
 on random dense and sparse graphs, using the results of Section~4 for the CL \eqref{eqn:csys},
 based on the fundamental theory reviewed in Section~2.
Finally, in Section~6,
 we give numerical simulations of the CKM \eqref{eqn:dsys}
 on complete simple graphs and on uniform random dense and sparse graphs,
 to demonstrate our theoretical results.

\section{Previous Fundamental Results}
We review the fundamental theoretical results from \cite{IY23,Y24a}
 in the setting of the CKM \eqref{eqn:dsys} and the CL \eqref{eqn:csys}.
For the proofs and additional details, 
 we refer to Section~2 and Appendices~A and B of \cite{IY23}, and to Section~2 of \cite{Y24a}.

We begin with the initial value problem (IVP) of the CL \eqref{eqn:csys}.
Let $g\in L^2(I)$ and write $\mathbf{u}:\Rset\to L^2(I)$ 
 for an $L^2(I)$-valued function on $\Rset$.
From Theorem~2.1 of \cite{IY23} we have the following.

\begin{thm}
\label{thm:2a}
There exists a unique solution $\mathbf{u}(t)\in C^1(\Rset;L^2(I))$
 to the IVP of \eqref{eqn:csys} with 
\begin{equation*}
u(0,x)=g(x).
\end{equation*}
Moreover, the solution depends continuously on $g$.
\end{thm}

We next state a theorem
 concerning the convergence of solutions of the CKM \eqref{eqn:dsys}
 to those of the CL \eqref{eqn:csys}.
The vector field of \eqref{eqn:dsys}
 is Lipschitz continuous with respect to the variables $u_i^n$, $i\in[n]$. 
Hence, by a standard result of ordinary differential equations
 (see, e.g., Theorem~2.1 of Chapter~1 of \cite{CL55}),
 the IVP of the CKM \eqref{eqn:dsys} admits a unique solution.
Let
\[
\mathbf{u}_n(t)=(u_1^n(t),u_2^n(t),\ldots,u_n^n(t))
\]
be such a solution. 
We identify it with the $L^2(I)$-valued step function
\begin{equation}
\mathbf{u}_n(t) = \sum_{i=1}^{n} u_i^n(t) \mathbf{1}_{I_i^n},
\label{eqn:un}
\end{equation}
where $\mathbf{1}_{I_i^n}$ is the characteristic function of $I_i^n$ for $i\in[n]$.
Let $\|\cdot\|$ denote the norm in $L^2(I)$.
From Theorem~2.3 of \cite{IY23}  (see also Theorem~2.2 of \cite{Y24a})
 we have the following.

\begin{thm}
\label{thm:2b}
If $\mathbf{u}_n(t)$ is the solution to the IVP
 of the CKM \eqref{eqn:dsys} with the initial condition
\[
\lim_{n\to\infty} \|\mathbf{u}_n(0)-\mathbf{u}(0)\|=0\quad\mbox{a.s.},
\]
 then for any $\tau>0$ we have
\[
\lim_{n \rightarrow \infty}\max_{t\in[0,\tau]}\|\mathbf{u}_n(t)-\mathbf{u}(t)\|=0
\quad\mbox{a.s.},
\]
where $\mathbf{u}(t)$ represents the solution to the IVP of the CL \eqref{eqn:csys}.
\end{thm}

We next discuss the stability of solutions to \eqref{eqn:dsys}
 and \eqref{eqn:csys}. 
Modifying the proof of Theorem~2.5 in \cite{IY23} slightly
 (see also Theorem~2.3 of \cite{Y24a}),
 we obtain the following.

\begin{thm}
\label{thm:2c}
Suppose that the CKM \eqref{eqn:dsys} and CL \eqref{eqn:csys}
 have solutions $\bar{\mathbf{u}}_n(t)$ and $\bar{\mathbf{u}}(t)$, respectively, such that
\begin{equation}
\lim_{n\to\infty}\|\bar{\mathbf{u}}_n(t)-\bar{\mathbf{u}}(t)\|=0\quad\mbox{a.s.}
\label{eqn:thm2c}
\end{equation}
for any $t\in[0,\infty)$.
Then the following hold$\,:$
\begin{enumerate}
\setlength{\leftskip}{-1.8em}
\item[\rm(i)]
If for any $\epsilon>0$, there exist $\delta_1>0$
 such that for $n>0$ sufficiently large,
 any solution $u_i^n(t)$, $i\in[n]$, to the CKM \eqref{eqn:dsys} with
\[
|u_i^n(0)-\bar{u}_i^n(0)|<\delta_1,
\quad i\in[n],
\] 
satisfies
\[
|u_i^n(t)-\bar{u}_i^n(t)|<\epsilon,
\quad i\in[n],\quad\mbox{a.s.}
\]
for any $t\in[0,\infty)$, then $\bar{\mathbf{u}}(t)$ is stable.
Moreover, if there exists $\delta_2>0$
 such that for $n>0$ sufficiently large,
 any solution $u_i^n(t)$, $i\in[n]$, to the CKM \eqref{eqn:dsys} with
\[
|u_i^n(0)-\bar{u}_i^n(0)|<\delta_2,\quad i\in[n],
\] 
converges to $\bar{u}_i^n(t)$, $i\in[n]$, uniformly in $n$ as $t\to\infty$,
 then $\bar{\mathbf{u}}(t)$ is asymptotically stable.
\item[\rm(ii)]
If $\bar{\mathbf{u}}(t)$ is stable, then for any $\epsilon,T>0$ there exists $\delta>0$
 such that for $n>0$ sufficiently large,
 if $\mathbf{u}_n(t)$ is a solution to the CKM \eqref{eqn:dsys} satisfying
\begin{equation*}
\|\mathbf{u}_n(0)-\bar{\mathbf{u}}_n(0)\|
 <\delta,
\end{equation*}
then
\begin{equation*}
\max_{t\in[0,T]}\|\mathbf{u}_n(t)-\bar{\mathbf{u}}_n(t)\|
 <\epsilon\quad\mbox{a.s.}
\end{equation*}
Moreover, if $\bar{\mathbf{u}}(t)$ is asymptotically stable, then
\begin{equation*}
\lim_{t\to\infty}\lim_{n\to\infty}
 \|\mathbf{u}_n(t)-\bar{\mathbf{u}}_n(t)\|=0\quad\mbox{a.s.}
\end{equation*}
\end{enumerate}
\end{thm}

\begin{rmk}\
\label{rmk:2a}
\begin{enumerate}
\setlength{\leftskip}{-1.8em}
\item[\rm(i)]
For the CL \eqref{eqn:csys},
 two solutions are identified
 if they differ only on a set of Lebesgue measure zero.
Hence, in Theorem~{\rm\ref{thm:2c}(ii)},
 we cannot simply say that $\bar{\mathbf{u}}_n(t)$ is stable or asymptotically stable
 in the CKM \eqref{eqn:dsys}.
Indeed, $\bar{\mathbf{u}}_n(t)$ may be unstable.
See also Remark~$2.8$ below.
\item[\rm(ii)]
The statements in Theorem~{\rm2.4(i)} of {\rm\cite{Y24a}} had a small error,
 which is corrected in Theorem~{\rm\ref{thm:2c}(i)} 
\end{enumerate}
\end{rmk}

We state the following corollary of Theorem~\ref{thm:2c},
 which does not require the existence of a solution $\bar{\mathbf{u}}_n(t)$ to the CKM \eqref{eqn:dsys}
 satisfying \eqref{eqn:thm2c} (see also Corollary~2.6 of \cite{Y24a}).
 
\begin{cor}
\label{cor:2a}
Suppose that the CL \eqref{eqn:csys} has a stable solution $\bar{\mathbf{u}}(t)$.
Then for any $\epsilon,T>0$ there exists $\delta>0$ such that for $n>0$ sufficiently large,
 if $\mathbf{u}_n(t)$ is a solution to the KM \eqref{eqn:dsys} satisfying
\[
\|\mathbf{u}_n(0)-\bar{\mathbf{u}}(0)\|<\delta,
\]
then
\[
\max_{t\in[0,T]}
 \|\mathbf{u}_n(t)-\bar{\mathbf{u}}(t)\| <\epsilon\quad\mbox{a.s.}
\]
Moreover, if $\bar{\mathbf{u}}(t)$ is asymptotically stable, then
\[
\lim_{t\to\infty}\lim_{n\to\infty}
 \|\mathbf{u}_n(t)-\bar{\mathbf{u}}(t)\|=0\quad\mbox{a.s.}
\]
\end{cor}

Thus, if the hypotheses of Corollary~\ref{cor:2a} hold,
 then for $n>0$ sufficiently large,
 $\bar{\mathbf{u}}(t)$ behaves as if it is an $($asymptotically$)$ stable solution
 in the CKM \eqref{eqn:dsys}.
We have the following from Theorems~2.7 and 2.9 of \cite{Y24a}.

\begin{thm}
\label{thm:2d}
Suppose that the hypothesis of Theorem~$\ref{thm:2c}$ holds.
Then the following hold$:$
\begin{enumerate}
\setlength{\leftskip}{-1.5em}
\item[\rm(i)]
If $\bar{\mathbf{u}}_n(t)$ is unstable a.s. for $n>0$ sufficiently large
 and no stable solution to the CKM \eqref{eqn:dsys} converges to $\mathbf{u}(t)$ a.s.
 as $n\to\infty$, then $\bar{\mathbf{u}}(t)$ is unstable$;$
\item[\rm(ii)]
If $\bar{\mathbf{u}}(t)$ is unstable, then so is $\bar{\mathbf{u}}_n(t)$ for $n>0$ sufficiently large.
\end{enumerate}
\end{thm}

\begin{thm}
\label{thm:2e}
If $\bar{\mathbf{u}}(t)$ is unstable,
 then for any $\epsilon,\delta>0$ there exist $\tau,N>0$ such that for $n>N$
\[
\|\mathbf{u}_n(\tau)-\bar{\mathbf{u}}(\tau)\|>\epsilon,\quad\mbox{a.s.}
\]
where $\mathbf{u}_n(t)$ is a solution to the KM \eqref{eqn:dsys} satisfying
\[
\|\mathbf{u}_n(0)-\bar{\mathbf{u}}(0)\|<\delta.
\]
\end{thm}

\begin{rmk}
\label{rmk:2c}
The hypothesis of Theorem~{\rm\ref{thm:2d}} alone does not imply that $\mathbf{u}(t)$ is unstable,
 even when $\mathbf{u}_n(t)$ is unstable for all sufficiently large $n$.
Moreover, it may happen that $\mathbf{u}(t)$ is asymptotically stable
 although $\mathbf{u}_n(t)$ is unstable for all sufficiently large $n$.
This behavior was proven to occur for the classical KM \eqref{eqn:dsys0}
 on complete simple graphs and its CL \eqref{eqn:csys0} previously in {\rm\cite{Y24a}}.
We will see below that it also occurs
 for the CKM \eqref{eqn:dsys} on complete graphs and its CL \eqref{eqn:csys}.
\end{rmk}

\section{Controlled Kuramoto Model on Complete Graphs}
In this section, we present our main results on synchronized solutions
 of the CKM \eqref{eqn:dsys} with the uniformly spaced natural frequencies \eqref{eqn:omegai}
 on complete graphs, i.e., $w_{ij}=p\in(0,1]$, $i,j\in[n]$, and $\alpha_n=1$.
We transform the CKM \eqref{eqn:dsys} into a tractable autonomous system,
 detect all its equilibria, the number of which is exactly $2^n$,
 and determine their bifurcation structure and stability.
A key result of this section is that the equilibrium
 corresponding to the synchronized solution \eqref{eqn:dsol} undergoes a saddle-node bifurcation,
 at which a pair of stable and unstable equilibria are born.
Moreover, we prove that, when $b_1>0$ and $K>0$ is sufficiently small,
 all other equilibria are unstable, and no equilibrium exists for $b_1=b_0=0$.
These results provide a complete picture
 of synchronization, bifurcations and stability for the CKM \eqref{eqn:dsys} on complete graphs.

Throughout this section, we take any integer $n\ge 3$ as the node number.
Let
\begin{equation}
v_i=u_i^n-V(t),\quad i\in[n].
\label{eqn:vi}
\end{equation}
Since we take $b_0=V_1$, Eq.~\eqref{eqn:dsys} becomes 
\begin{equation}
\dot{v}_i= 
 (2i-n-1)\nu+\frac{pK}{n}\sum_{j=1}^n\sin\left(v_j-v_i\right)-b_1\sin v_i,\quad i\in[n],
\label{eqn:dsys1}
\end{equation}
where $\nu=a/2n$.
If Eq.~\eqref{eqn:dsys1} has an equilibrium at $v_i=\zeta_i$, $i\in[n]$, then
\begin{equation}
u_i^n=\zeta_i+V(t),\quad
i\in[n].
\label{eqn:dsyssol}
\end{equation}
is a solution to the CKM \eqref{eqn:dsys}.
Let $v=(v_1,\ldots,v_n)$ and $\zeta=(\zeta_1,\ldots,\zeta_n)$.
If the equilibrium $v=\zeta$ is (asymptotically) stable,
 then so is the solution \eqref{eqn:dsyssol},
 and if a bifurcation of $v=\zeta$ occurs,
 then so does that of \eqref{eqn:dsyssol}.
Hence, we analyze the system \eqref{eqn:dsys1}
 instead of the CKM \eqref{eqn:dsys} below.
Our approaches are modifications of ones used in \cite{Y24a},
 in which the classical KM given by \eqref{eqn:dsys0} with $p=1$ 
 was analyzed when $n$ is any odd number with $n\geq 3$,
 although the modifications are not straightforward.

\subsection{Equilibria}
Let $\sigma=\{\sigma_i\}_{i=1}^n$ be a sequence of length $n$
 with $\sigma_i\in\{-1,1\}$, $i\in[n]$, 
 and let
\[
\Sigma_n = 
\left\{\sigma=\{\sigma_i\}_{i=1}^n\mid\sigma_i\in\{-1,1\},i\in[n]\right\}.
\]
For each $\sigma\in\Sigma_n$, we define $C_\D^{\sigma}$ such that 
\begin{equation}
C_\D^{\sigma} = 
\frac{1}{n}\sum_{i=1}^{n}\sigma_i\sqrt{1-\left(\frac{(2i-n-1)\nu}{pKC_\D^{\sigma}+b_1}\right)^2},
\label{eqn:CDj}
\end{equation}
and write
\begin{equation}
v_i^\sigma = 
\begin{cases}
\phi_i & \mbox{if $\sigma_i=1$;}\\
\pi-\phi_i & \mbox{if $\sigma_i=-1$ and $\phi_i\ge0$;}\\
-\phi_i-\pi & \mbox{if $\sigma_i=-1$ and $\phi_i<0$,}
\end{cases}
\label{eqn:vis}
\end{equation}
where
\begin{equation}
\phi_i = 
\arcsin\left(\frac{(2i-n-1)\nu}{pKC_\D^{\sigma}+b_1}\right).
\label{eqn:phi}
\end{equation}
As in Theorem~4.1 of \cite{Y24a}, we can prove the following
 on the existence of equilibria in \eqref{eqn:dsys1}.

\begin{thm}\
\label{thm:3a}
\begin{enumerate}
\setlength{\leftskip}{-1.5em}
\item[\rm(i)]
For each $\sigma\in\Sigma_n$,
 $v=v^\sigma\in\Tset^{n}:=\prod_{i=1}^{n}\Sset^1$
 gives an equilibrium in \eqref{eqn:dsys1}, 
 when $C_\D^{\sigma}$ satisfies \eqref{eqn:CDj}.
Moreover, no other equilibrium exists in \eqref{eqn:dsys1} for $b_1\neq 0$.
\item[\rm(ii)]
Fix $b_1,a,pK>0$ and let $i\le n/2$.
If the equilibrium $v=v^\sigma$ with $\sigma_i=-1$ and $\sigma_{n-i+1}=1$ exists, 
 then so does $v=v^{\bar{\sigma}}$ with $\bar{\sigma}_i=1$, $\bar{\sigma}_{n-i+1}=-1$
 and $\bar{\sigma}_j=\sigma_j$, $j\in[n]\setminus\{i,n-i+1\}$, and vice versa.
\end{enumerate}
\end{thm}

\begin{proof}
Let $\zeta=(\zeta_1,\cdots,\zeta_n)$ denote an equilibrium of \eqref{eqn:dsys1}. 
Recall that $b_0=V_1$.
We begin with the following lemma.

\begin{lem}
\label{lem:3a}
We have
\begin{equation*}
\sum_{i=1}^{n}\sin\zeta_i=0.
\end{equation*}
\end{lem}

\begin{proof}
When $v_i=\zeta_i$, $i\in[n]$,
 the sum of the right-hand sides of \eqref{eqn:dsys1} is zero.
This yields the desired result.
\end{proof}

Let 
\begin{equation*}
C_{\zeta}=\frac{1}{n}\sum_{i=1}^{n}\cos\zeta_i.
\end{equation*}
Using Lemma~$\ref{lem:3a}$ and \eqref{eqn:dsys1}, we have 
\begin{equation}
(2i-n-1)\nu-\left(pKC_{\zeta}+b_1\right)\sin\zeta_i=0.
\end{equation}
Hence, if $\phi_i\geq0$ (resp. $\phi_i<0$),
 then $\zeta_i=\phi_i$ or $\pi-\phi_i$ (resp. $\zeta_i=\phi_i$ or $-\pi-\phi_i$), $i\in[n]$,
 where $\phi_i$ is given by \eqref{eqn:phi} with $C_\D^{\sigma}=C_{\zeta}$.
This yields part~(i).
By \eqref{eqn:CDj}, we have $C_\D^{\sigma}=C_\D^{\bar{\sigma}}$,
 so that $\phi_i^\sigma=\phi_i^{\bar{\sigma}}$, $i\in[n]$.
This yields part~(ii).
\end{proof}

For $\sigma\in\Sigma_n$, we define
\[
\chi^\sigma(\xi) = 
\frac{1}{n}\sum_{i=1}^{n}\sigma_i\sqrt{1-\left(\frac{2i-n-1}{n-1}\xi\right)^2},
\]
where $\xi\in[0,1]$.
Letting $\xi=(n-1)\nu/|pKC_D^{\sigma}+b_1|$, 
 we rewrite \eqref{eqn:CDj} as
\begin{equation}
\frac{pK}{b_1} = 
\bar{\chi}^\sigma(\xi) := 
\frac{\xi}{\pm\beta-\xi\chi^\sigma(\xi)},\quad
\beta=\frac{(n-1)\nu}{pK}>0,
\label{eqn:con}
\end{equation}
where the upper and lower signs, respectively,
 correspond to the cases $pK\chi^{\sigma}(\xi)+b_1>0$ and $pK\chi^{\sigma}(\xi)+b_1<0$.
The following corollary is an immediate consequence of Theorem~\ref{thm:3a}(i).

\begin{cor}
\label{cor:3a}
Fix the value of $pK>0$.
If $\xi\in(0,1]$ satisfies \eqref{eqn:con} 
 for $\sigma\in\Sigma_n$ and $b_1>0$,
 then $v^\sigma$ given by \eqref{eqn:vis} with 
\begin{equation*}
\phi_i = \pm\arcsin\left(\frac{2i-n-1}{n-1}\xi\right),
\end{equation*}
instead of \eqref{eqn:phi}, is an equilibrium in \eqref{eqn:dsys1}, 
 where the upper or lower signs, respectively,
 correspond to the cases $pK\chi^{\sigma}(\xi)+b_1>0$ and $pK\chi^{\sigma}(\xi)+b_1<0$.
\end{cor}

\begin{rmk}\
\label{rmk:3a}
\begin{enumerate}
\setlength{\leftskip}{-1.4em}
\item[\rm(i)]
From Theorem~$\ref{thm:3a}$
 we easily see that the number of equilibria in \eqref{eqn:dsys1},
 which correspond to synchronized solutions of the form \eqref{eqn:dsys}, is exactly $2^n$.
\item[\rm(ii)]
We easily see that
 if $\sigma\neq\hat{\sigma}$, $\xi\neq 1$ and $b_1\neq 0$,
 then $v^\sigma\neq v^{\hat{\sigma}}$.
\item[\rm(iii)]
From the result of {\rm\cite{IY23,Y24a}} we see that
 if $b_1=0$, then there exists a one-parameter family of equilibria in \eqref{eqn:dsys1}.
\end{enumerate}

\end{rmk}

\begin{figure}
\includegraphics[scale=0.58]{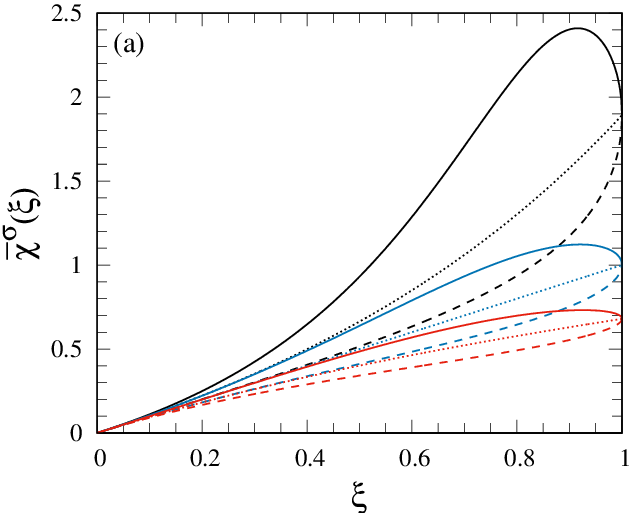}\
\includegraphics[scale=0.58]{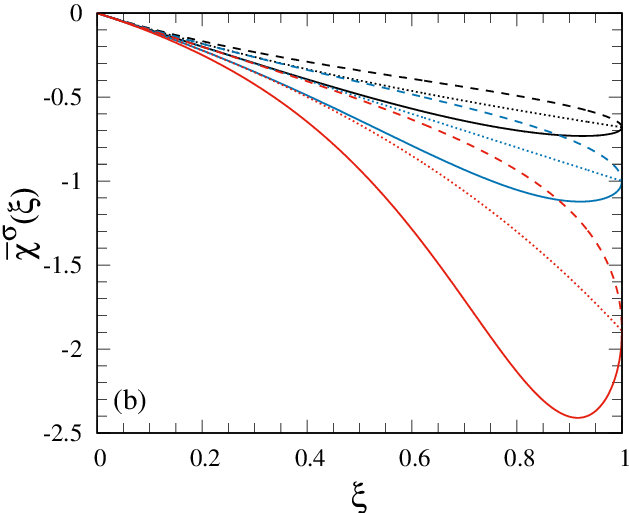}\\[1ex]
\includegraphics[scale=0.58]{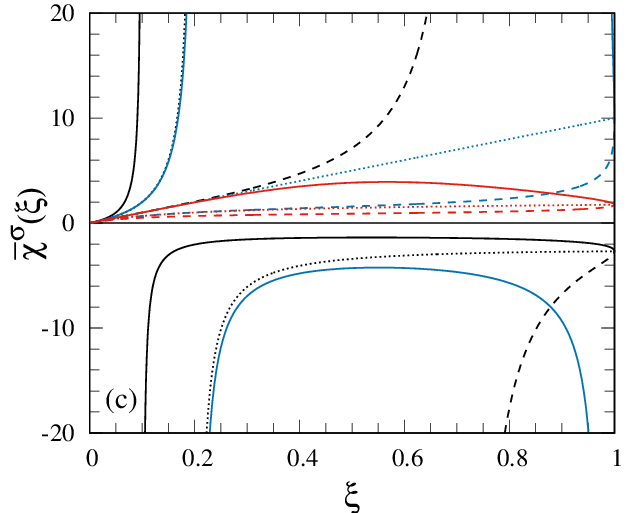}
\includegraphics[scale=0.58]{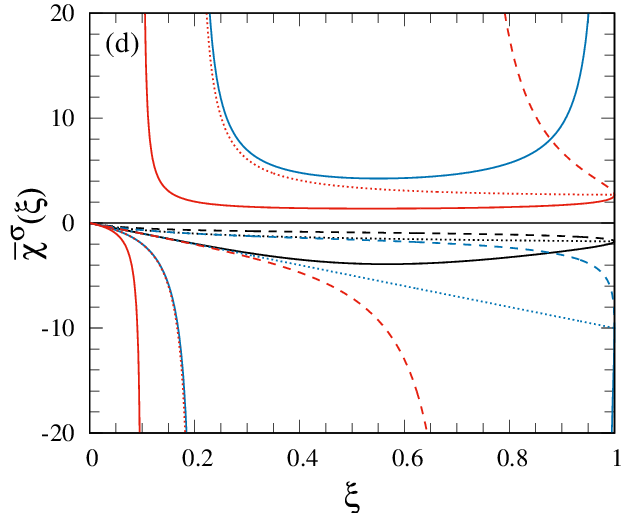}
\caption{Function $\bar{\chi}^\sigma(\xi)$ for $n=4$:
(a) and (b) $\beta=1$; (c) and (d) $0.1$.
In plates~(a) and (c) (resp. plates~(b) and (d)) the sign `$+$' (resp. `$-$') is taken
 in  \eqref{eqn:con}.
See the text for more details.
\label{fig:3a}}
\end{figure}

Figure~\ref{fig:3a} shows the graph of $\bar{\chi}^\sigma(\xi)$ for $n=4$.
Solid lines indicate the cases
 where $\bar{\chi}^\sigma(\xi)$ has a local maximum or minimum.
In the remaining cases, dotted and dashed lines are used
 for $\sigma_1=\sigma_n$ and $\sigma_1\neq\sigma_n$, respectively.
The black and red lines represent the graph
 for $(\sigma_2,\sigma_3)=(1,1)$ and  $(-1,-1)$, respectively,
 while the blue line for $(\sigma_2,\sigma_3)=(1,-1)$ or $(-1,1)$.
We see that $\bar{\chi}^\sigma(\xi)$ is bounded if $\beta>0$ is sufficiently large 
 and that it is unbounded if $\beta>0$ is sufficiently small.
We prove the following.

\begin{prop}
\label{prop:3a}
Suppose that
\begin{equation}
\beta=\frac{(n-1)\nu}{pK}
 >\max_{\xi\in[0,1]}\xi\chi^\sigma(\xi)\quad\mbox{with $\sigma_i=1$, $i\in[n]$}.
\label{eqn:prop3a}
\end{equation}
Then the system~\eqref{eqn:dsys1} has no equilibrium when $b_1=0$.
\end{prop}

\begin{proof}
Assume that the hypothesis of the proposition holds.
Then for any $\sigma\in\Sigma_n$
\begin{equation}
\beta>\max_{\xi\in[0,1]}|\xi\chi^\sigma(\xi)|,
\label{eqn:prop3a'}
\end{equation}
so that $\bar{\chi}^\sigma(\xi)$ is bounded on $[0,1]$.
This implies the desired result
 since the relation \eqref{eqn:con} does not hold near $b_1=0$.
\end{proof}

\subsection{Bifurcations}
We next state bifurcation results for the equilibria in \eqref{eqn:dsys1}
 given by Theorem~\ref{thm:3a} and Corollary~\ref{cor:3a}.
Regarding $b_1$ as a control parameter,
 a branch of equilibria is determined by \eqref{eqn:con} for each $\sigma\in\Sigma_n$.
As in Theorem~5.1 of \cite{Y24a},
 we have the following.

\begin{thm}\
\label{thm:3b}
Fix the values of $a,pK>0$ and choose $b_1$ as a control parameter.
Then the following hold.
\begin{enumerate}
\setlength{\leftskip}{-1.5em}
\item[\rm(i)]
The equilibrium $v^{\sigma}$ suffers a supercritical 
 $($resp. subcritical$)$ saddle-node bifurcation at
\begin{equation}
b_1=\frac{pK}{\bar{\chi}^{\sigma}(\xi_0)}
\label{eqn:thm3b1}
\end{equation}
in \eqref{eqn:dsys1} if and only if 
 $\bar{\chi}^{\sigma}(\xi)$ has a local maximum $($resp. a local minimum$)$ at $\xi=\xi_0$ on $(0,1)$ .
In particular, if $\bar\chi^\sigma(\xi)$ is bounded on $[0,1]$ and
\begin{equation}
\sigma_1,\sigma_n=1,\quad
\bar{\chi}^\sigma(1)\ge 0\quad
(\mbox{resp. }
\sigma_1,\sigma_n=-1,\quad
\bar{\chi}^\sigma(1)\le 0),
\label{eqn:thm3b2}
\end{equation}
then a supercritical $($resp. subcritical$)$ saddle-node bifurcation occurs.
Moreover, if $b_1>0$, $\sigma_i=1$, $i\in[n]$,
 and condition~\eqref{eqn:prop3a} holds,
 then $\bar{\chi}^{\sigma}(\xi)$ has a unique local maximum 
 and no local minimum, and $v^\sigma$ suffers only one 
 saddle-node bifurcation.
\item[\rm(ii)]
Let $v^{\sigma^{\pm\pm}}$ be four equilibria 
 in \eqref{eqn:dsys1} such that 
\[
\sigma_i^{++}=\sigma_i^{--}=\sigma_i^{+-}=\sigma_i^{-+},\quad
i\neq 1,n,
\]
and
\[
\sigma_1^{++},\sigma_1^{+-}=1,\quad
\sigma_1^{-+},\sigma_1^{--}=-1,\quad
\sigma_n^{++},\sigma_n^{-+}=1,\quad
\sigma_n^{+-},\sigma_n^{--}=-1.
\]
If $\bar{\chi}^\sigma(1)\neq 0$, then a pitchfork bifurcation where
 $v^{\sigma^{++}}$ changes to $v^{\sigma^{--}}$
 and where $v^{\sigma^{+-}}$ and $v^{\sigma^{-+}}$ are born 
 occurs at
\begin{equation}
b_1=\frac{pK}{\bar{\chi}^{\sigma}(1)},
\label{eqn:thm3b3}
\end{equation}
where any of $\sigma^{\pm\pm}$ may be chosen as $\sigma$.
Moreover, the bifurcation is super- or subcritical if
\begin{equation}
\frac{\d\bar{\chi}^{\sigma}}{\d\xi}(1)\quad
\mbox{with $\sigma=\sigma^{+-}$ and $\sigma^{-+}$}
\label{eqn:thm3b4}
\end{equation}
is positive or negative,
 where any of $\sigma^{+-}$ and $\sigma^{-+}$ may be chosen as $\sigma$.
\end{enumerate}
\end{thm}

\begin{proof}
We first recall from Corollary~\ref{cor:3a} that
 $v^\sigma$ is an equilibrium in \eqref{eqn:dsys1}
 if $\xi\in(0,1)$ satisfies \eqref{eqn:con}.
Hence, if $\bar{\chi}^\sigma(\xi)$ has a local maximum 
 (resp. a local minimum) at $\xi=\xi_0$ on $(0,1)$, 
 then any value of $\xi$ does not satisfy \eqref{eqn:con}
 for values of $b_1$ that are less (resp. greater) than and close to \eqref{eqn:thm3b1} 
 but two values of $\xi$ satisfy it for values of $b_1$ that are 
 greater (resp. less) than and close to \eqref{eqn:thm3b1},
 so that a supercritical $($resp. subcritical$)$ saddle-node bifurcation of $v^\sigma$ occurs there.
If $\bar\chi^\sigma(\xi)$ is bounded on $[0,1]$
 and condition \eqref{eqn:thm3b2} holds,
 then $\bar{\chi}^\sigma(\xi)$ has a local maximum (resp. a local minimum),
 since $\bar{\chi}^\sigma(0)=0$ and
\begin{equation}
\frac{\d\bar{\chi}^{\sigma}}{\d\xi}(\xi)
=\frac{\displaystyle\pm\beta+\xi^2\frac{\d \chi^{\sigma}}{\d \xi}(\xi)}
 {(\pm\beta-\xi\chi^{\sigma}(\xi))^2}\to-\infty\quad
(\mbox{resp. }+\infty)
\label{eqn:thm3b5}
\end{equation}
as $\xi\to 1$.
Note that
\begin{align*}
&
\frac{\d \chi^{\sigma}}{\d \xi}(\xi)\\
&
=-\frac{1}{n}\sum_{i=1}^{n}\sigma_i\left(\frac{2i-n-1}{n-1}\right)^2\xi
\Bigg/\sqrt{1-\left(\frac{2i-n-1}{n-1}\xi\right)^2}\to-\infty\quad
(\mbox{resp. }+\infty)
\end{align*}
as $\xi\to 1$ under the conditions.

Assume that $b_1>0$, $\sigma_i=1$, $i\in[n]$,
 and condition~\eqref{eqn:prop3a} holds.
Since $pK\chi^\sigma(\xi)+b_1>0$, we have
\begin{equation*}
\bar{\chi}^\sigma(\xi) = \frac{\xi}{\beta-\xi\chi^\sigma(\xi)}>0
\end{equation*}
and
\begin{equation}
\frac{\d \bar{\chi}^\sigma}{\d \xi}(\xi) = 
\frac{\displaystyle\beta-\frac{\xi}{n}\sum_{i=1}^{n}\left(\frac{2i-n-1}{n-1}\xi\right)^2
 \Bigg/ \sqrt{1-\left(\frac{2i-n-1}{n-1}\xi\right)^2}}%
 {(\beta-\xi\chi^\sigma(\xi))^2}.
\label{eqn:thm3b6}
\end{equation}
The numerator in \eqref{eqn:thm3b6} is monotonically decreasing on $(0,1)$
 and positive at $\xi=0$ and goes to  $-\infty$ as $\xi\to1$
 while the denominator is always positive.
Hence, $\bar{\chi}^\sigma(\xi)$ has a unique local maximum
 at which its derivative has a unique zero,
 and it has no local minimum.
This proves part~(i).

On the other hand, assume that $\bar{\chi}^\sigma(1)\neq 0$.
Then at $\xi=1$, the four equilibria $v^{\sigma^{\pm\pm}}$ coincide
 and the corresponding functions $\bar{\chi}^{\sigma^{\pm\pm}}(\xi)$ have the same value.
From \eqref{eqn:thm3b5} we see that
\[
\frac{\d\bar{\chi}^{\sigma^{++}}}{\d\xi}(\xi)\to-\infty,\quad
\frac{\d\bar{\chi}^{\sigma^{--}}}{\d\xi}(\xi)\to+\infty
\]
as $\xi\to 1$.
Moreover, by Theorem~\ref{thm:3a}(ii),
 $v^{\sigma^{+-}}$ and $v^{\sigma^{-+}}$ exist in a pair
 for values of $b_1$ greater or less than, but close to, the value in \eqref{eqn:thm3b3}.
The alternative is decided by the sign of \eqref{eqn:thm3b4},
 whose value is the same for $\sigma=\sigma^{+-}$ and $\sigma=\sigma^{-+}$.
This proves part~(ii).
\end{proof} 

From Theorem~\ref{thm:3b}
 we can also estimate the numbers of saddle-node and pitchfork bifurcations
 if condition~\eqref{eqn:prop3a} holds, as in Remark~5.2 and Proposition~5.3 of \cite{Y24a}.

\begin{prop}
\label{prop:3b'}
Suppose that condition~\eqref{eqn:prop3a} holds.
Then both the numbers of saddle-node and pitchfork bifurcations are at least $2^{n-2}$.
\end{prop}

\begin{proof}
We assume that condition~\eqref{eqn:prop3a} holds.
From the proof of Proposition~\ref{prop:3a}
 we see that $|\bar{\chi}^\sigma(\xi)|$ is bounded on $[0,1]$
 as well as $\bar{\chi}^\sigma(1)\neq 0$ for any $\sigma\in\Sigma_n$.
Hence, the number of $\sigma\in\Sigma_n$ satisfying condition \eqref{eqn:thm3b2} is $2^{n-2}$,
 so that by Theorem~\ref{thm:3b}(i), $2^{n-2}$ saddle-node bifurcations occur.
Moreover, it follows from Theorem~\ref{thm:3b}(ii)
 that $2^{n-2}$ pitchfork bifurcations occur since $\bar{\chi}^\sigma(1)\neq 0$ for any $\sigma\in\Sigma_n$.
\end{proof}

As in Proposition~4.4 of \cite{Y24a}, we also have the following.

\begin{prop}
\label{prop:3b}
For any $\sigma\in\Sigma_n$,
 the equilibrium $v^\sigma$ suffers no Hopf bifurcation.
\end{prop}

\begin{proof}
Let $A$ denote the Jacobian matrix for the vector field of \eqref{eqn:dsys1}.
We compute each element of $A$ as
\begin{equation}
A_{ij} = 
\begin{cases}
\displaystyle
-\frac{pK}{n}\sum_{j=1,j\neq i}^{n}\cos\left(v_j-v_i\right)
-b_1\cos v_i & \mbox{if $i=j$}; \\
\displaystyle
\frac{pK}{n}\cos\left(v_j-v_i\right) & \mbox{if $i\neq j$}.
\end{cases}
\label{eqn:jmat1}
\end{equation}
Hence, the matrix $A$ has only real eigenvalues
 since it is symmetric.
This proves the proposition.
\end{proof}

\begin{rmk}
\label{rmk:3c}
From Remark~{\rm\ref{rmk:3a}(ii)} and Proposition~$\ref{prop:3b}$ we see that
 for $b_1\neq 0$, no bifurcation of equilibria for \eqref{eqn:dsys1} occurs
 except for those detected in Theorem~$\ref{thm:3b}$.
\end{rmk}

\subsection{Stability}

We finally state a theorem on the stability of the equilibria in \eqref{eqn:dsys1}
 given by Theorem~\ref{thm:3a} and Corollary~\ref{cor:3a}.
As in Theorem~6.1 of \cite{Y24a},
 we prove the following.
 
\begin{thm}
\label{thm:3c}
Suppose that $b_1>0$ and condition~\eqref{eqn:prop3a} holds.
Then the following hold$:$
\begin{enumerate}
\setlength{\leftskip}{-1.5em}
\item[\rm(i)]
The equilibrium $v^\sigma$ with $\sigma_i=1$ for all 
 $i\in[n]$ is asymptotically stable 
 if $\xi
 <\xi_0$
 and unstable if $\xi>\xi_0$, where $\xi_0$ is 
 the unique local maximum of $\bar{\chi}^{\sigma}(\xi)$ 
 detected in Theorem~{\rm\ref{thm:3b}(i)}.
 
\item[\rm(ii)]
The equilibrium $v^{\sigma}$ is unstable if $\sigma_i=-1$ 
 for some $i\in[n]$.
\end{enumerate}
\end{thm}

\begin{proof}
Assume that the hypotheses of the theorem are satisfied.
We fix $\sigma\in\Sigma_n$
 and follow the equilibrium branch $v^{\sigma}$
 for $\bar{\chi}^{\sigma}(\xi)$ as $\xi$ increases in $(0,1)$.
For clarity, we write this branch as $v^{\sigma}(\xi)$
 and denote by $f(v;b_1)$ the vector field of \eqref{eqn:dsys1}.
By Proposition~\ref{prop:3a},
 the branch $b_1=pK/\bar{\chi}^{\sigma}(\xi)$ does not meet $b_1=0$.
We begin with the following result corresponding to Lemma~6.2 of \cite{Y24a}.

\begin{lem}
\label{lem:3b}
Suppose that the multiplicity of the zero eigenvalue
 of the Jacobian matrix $\D_vf(v^\sigma(\xi);b_1)$ 
 with $\bar{\chi}^{\sigma}(\xi)$, $\sigma\in\Sigma_n$, changes at $\xi=\xi_*$ 
 as $\xi$ increases in $(0,1)$.
Then $\bar{\chi}^{\sigma}(\xi)$ has an extremum at $\xi=\xi_*$
 and the multiplicity changes by one.
\end{lem}

\begin{proof}
Since $\D_vf(v^{\sigma}(\xi);b_1)$ is symmetric
 as seen in the proof of Proposition~\ref{prop:3b},
 all its eigenvalues are real and semisimple.
Hence, an eigenvalue can change sign only by passing through zero.

Assume first that $\D_vf(v^{\sigma}(\xi);b_1)$ has a simple zero eigenvalue
 when Eq.~\eqref{eqn:con} holds with $\xi=\xi_*\in(0,1)$.
Let $\bar e\in\Rset^n$ be an eigenvector for the zero eigenvalue.
Then there exists a one-dimensional center manifold \cite{GH83,K04,W03}
 for the equilibrium $v^{\sigma}(\xi_*)$,
 and the restriction of \eqref{eqn:dsys1} to it is written as
\begin{equation}
\dot{v}_\c=c_1 v_\c^j+\cdots+c_2(b_1-b_{1*})+\cdots,\quad
v_\c\in\Rset,
\label{eqn:lem3b1}
\end{equation}
where $b_{1*}=pK/\bar{\chi}^{\sigma}(\xi_*)$, 
 $c_1,c_2\in\Rset$ are constants, and $j>1$ is an even integer.
Note that the eigenvalue does not  change its sign if $j$ is odd, by \eqref{eqn:lem3b1}.
If $\bar{e}$ is linearly independent of $(\d v^\sigma/\d\xi)(\xi_*)$,
 then another equilibrium branch would pass through $v^{\sigma}(\xi_*)$.
This is impossible, since $v^\sigma(\xi_*)$ is isolated for $\xi\in(0,1)$ by Remark~\ref{rmk:3a}(ii).
Thus, we may choose
\begin{equation}
\bar e=\frac{\d v^\sigma}{\d\xi}(\xi_*).
\label{eqn:lem3b2}
\end{equation}
Noting \eqref{eqn:con} and differentiating the relation
\[
f\left(v^\sigma(\xi); \frac{pK}{\bar{\chi}^{\sigma}(\xi)}\right)\equiv 0
\]
with respect to $\xi$ at $\xi=\xi_*$, we obtain
\begin{align*}
&~\frac{\d}{\d\xi}f\left(v^{\sigma}(\xi);
\frac{pK}{\bar{\chi}^{\sigma}(\xi)}\right)\bigg|_{\xi=\xi_*}\\
&=\D_v f(v^{\sigma}(\xi_*);b_{1*})\frac{\d v^\sigma}{\d\xi}(\xi_*)
-\frac{\partial f}{\partial b_1}(v^{\sigma}(\xi_*);b_{1*})
\frac{pK}{\left(\bar{\chi}^{\sigma}(\xi_*)\right)^2}
\frac{\d\bar{\chi}^{\sigma}}{\d\xi}(\xi_*) =0.
\end{align*}
Since
\[
\frac{\partial f}{\partial b_1}(v^{\sigma}(\xi_*);b_{1*})
=-(\sin v_1^\sigma(\xi_*),\ldots,\sin v_n^\sigma(\xi_*))^\T\neq 0,
\]
where the superscript `{\scriptsize$\T$}' represents the transpose operator,
 it follows from \eqref{eqn:lem3b2} that
\[
\frac{\d\bar{\chi}^{\sigma}}{\d\xi}(\xi_*)=0.
\]
Thus, an eigenvalue can become zero
 only at a critical point of $\bar{\chi}^{\sigma}(\xi)$.

Moreover, $\D_v f(v^{\sigma}(\xi);b_1)$ cannot have a non-simple zero eigenvalue.
Indeed, such an eigenvalue would contradict the isolation of $v^{\sigma}(\xi)$ for $\xi\in(0,1)$,
 as in the argument above.
This yields the desired result.
\end{proof}

We need another result for the proof of Theorem~\ref{thm:3c}.
Fix $\sigma\in\Sigma_n$.
Taking the limit $\xi=(n-1)\nu/|pKC_\D^{\sigma}+b_1|\to 0$, 
 we have $b_1\to+\infty$ and $\phi_i\to 0$ in \eqref{eqn:phi}, 
 so that by \eqref{eqn:vis}
\begin{equation*}
v_i^{\sigma}\to 
\begin{cases}
0 & \mbox{if $\sigma_i=1$;}\\
\pi\mbox{ or }-\pi & \mbox{if $\sigma_i=-1$,}
\end{cases}
\end{equation*}
since the branch $b_1=pK/\bar{\chi}^\sigma(\xi)$ does not intersect $b_1=0$.
Let $n_+$ and $n_-$ be, respectively,
 the numbers of $\sigma_i=1$ and $-1$, $i\in[n]$.
The Jacobian matrix $A=\D_vf(v^\sigma(\xi);b_1)$ is written as
\begin{equation*}
A=b_1 A_0+O(1),\quad
b_1\to\infty,
\end{equation*}
as $\xi\to0$, where $A_0$ is an $n\times n$ diagonal matrix
 whose $n_-$ and $n_+$ diagonal elements are $1$ and $-1$, respectively.
We immediately obtain the following.

\begin{lem}
\label{lem:3c}
Let $\xi>0$ be sufficiently small.
Then the numbers of positive and negative eigenvalues of the matrix $A$
 are $n_-$ and $n_+$, respectively.
In particular, the matrix $A$ has no positive eigenvalue
 if and only if $n_+=n$, i.e., $n_-=0$.
\end{lem}

We now prove part~(i).
Let $\sigma_i=1$, $i\in[n]$, i.e., $n_-=0$.
Since by Theorem~\ref{thm:3b}(i)
 $\bar{\chi}^\sigma(\xi)$ has an extremum only at $\xi=\xi_0$,
 it follows from Lemma~\ref{lem:3b} 
 that only one eigenvalue of $\D_v f(v;b_1)$ with $b_1=pK/\bar{\chi}^\sigma(\xi)$
 changes its sign once at $\xi=\xi_0$ when $\xi$ increases from zero to one.
Since all of its eigenvalues are negative near $\xi=0$ by Lemma~\ref{lem:3c},
 we obtain the desired result.

We next prove part~(ii).
We compute
\[
\frac{\d v_i^\sigma}{\d\xi}(\xi)
=\pm\sigma_i\left(\frac{2i-n-1}{n-1}\right)\bigg/
 \sqrt{1-\left(\frac{2i-n-1}{n-1}\xi\right)^2}
\]
and
\[
\frac{\partial f_i}{\partial b_1}(v^\sigma(\xi);b_1)
=\mp\sigma_i\left(\frac{2i-n-1}{n-1}\xi\right)
\]
for $i\in[n]$ when $b_1=pK/\bar{\chi}^\sigma(\xi)$,
 where $f_i(v;b_1)$ is the $i$th element of $f(v;b_1)$
 and the upper and lower signs, respectively,
 correspond to the cases $pK\chi^\sigma(\xi)+b_1>0$ and $pK\chi^\sigma(\xi)+b_1<0$.

\begin{figure}[t]
\includegraphics[scale=0.8]{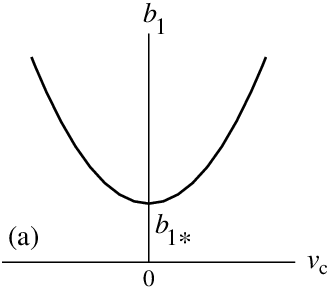}\qquad
\includegraphics[scale=0.8]{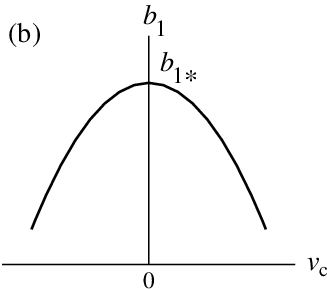}
\caption{Bifurcation diagrams of equilibria in \eqref{eqn:lem3b1}:
(a) $c_1/c_2<0$; (b) $c_1/c_2>0$.
\label{fig:3b}}
\end{figure}

Let $\xi=\xi_*$ be a zero of $(\d\bar{\chi}^\sigma/\d\xi)(\xi)$
 and assume that $b_{1*}=pK/\bar{\chi}^\sigma(\xi_*)>0$ is bounded.
Then $\D_vf(v^\sigma(\xi_*);b_{1*})$ has a simple zero eigenvalue
 and $(\d v^\sigma/\d\xi)(\xi_*)$ is its eigenvector.
There exists a one-dimensional center manifold for $v^\sigma(\xi_*)$,
 on which the system \eqref{eqn:dsys1} takes the form \eqref{eqn:lem3b1}.
See the proof of Lemma~\ref{lem:3b}.
In addition, 
 $c_1/c_2$ is negative (resp. positive)
 when $\bar{\chi}^\sigma(\xi)$ has  a local maximum (resp. a local minimum) at $\xi=\xi_*$,
 since by Theorem~\ref{thm:3b}(i)
 a supercritical (resp. subcritical) saddle-node bifurcation occurs.
See Fig.~\ref{fig:3b}.

Moreover, substituting $v=(\d v^\sigma/\d\xi)(\xi_*)v_\c+v^\sigma(\xi_*)$
 into \eqref{eqn:dsys1}, we take the inner product of the resulting equation
 with $(\d v^\sigma/\d\xi)(\xi_*)$ and obtain
\[
c_2=\frac{\d v^\sigma}{\d\xi}(\xi_*)
 \cdot\frac{\partial f}{\partial b_1}(v^\sigma(\xi_*);b_{1*})\bigg/
 \biggl| \frac{\d v^\sigma}{\d\xi}(\xi_*)\biggr|^2<0
\]
where the dot `$\cdot$' represents the standard inner product, since
\begin{align*}
&\frac{\d v^\sigma}{\d\xi}(\xi_*)
 \cdot\frac{\partial f}{\partial b_1}(v^\sigma(\xi_*);b_{1*})\\
&=-\sum_{i=1}^{n}\left(\frac{2i-n-1}{n-1}\right)^2\xi_*\bigg/
 \sqrt{1-\left(\frac{2i-n-1}{n-1}\xi_*\right)^2}<0.
\end{align*}
We also have
\[
v_\c=\frac{\d v^\sigma}{\d\xi}(\xi_*)\cdot(v-v^\sigma(\xi_*))
 \bigg/\left|\frac{\d v^\sigma}{\d\xi}(\xi_*)\right|^2,
\]
so that $v_\c$ is negative or positive for $v=v^\sigma(\xi)$ near $\xi=\xi_*$,
 depending on whether $\xi$ is less or greater than $\xi_*$.
Hence, we see via \eqref{eqn:lem3b1} that
 a negative (resp. positive) eigenvalue of $\D_v f(v^\sigma(\xi);b_1)$
 with $b_1=pK/\bar{\chi}^\sigma(\xi)$
 becomes positive (resp. negative) 
 when $\bar{\chi}^\sigma(\xi)$ has a local maximum (resp. a local minimum) at $\xi=\xi_*$,
 i.e., $c_1>0$ (resp. $c_1<0$).
See also Fig.~\ref{fig:3b}.

Fix $\sigma\in\Sigma_n$ and let $\xi$ vary from $0$ to $1$.
We see that $\bar{\chi}^\sigma(\xi)$ must take a local maximum before taking a local minimum 
 since it does not become zero on $(0,1]$ by \eqref{eqn:prop3a'}.
Hence,  if $\sigma_i=-1$ for some $i\in[n]$,
 then by Lemma~\ref{lem:3c} $v^\sigma(\xi)$ is unstable near $\xi=0$
 and $v^\sigma(\xi)$ continues to be unstable when changing $\xi$ from $0$ to $1$,
 since the number of positive eigenvalues of $\D_v f(v^\sigma(\xi);b_1)$ does not vanish.
Thus, we complete the proof of Theorem~\ref{thm:3c}.
\end{proof}

\begin{rmk}
\label{rmk:3d}
From the proof of Theorem~{\rm\ref{thm:3c}} we see that
 the equilibrium $v^\sigma$ with $\sigma_i=-1$, $i\in[n]$, is stable
 for $b_1<0$ with $|b_1|\gg 1$.
\end{rmk}

By the relation \eqref{eqn:vi} we have
\begin{equation}
u_i^n = v_i^{\sigma}+V(t)
\label{eqn:dsol0}
\end{equation}
as a synchronized solution in the CKM \eqref{eqn:dsys} for each $\sigma\in\Sigma_n$.
In particular, when $\sigma_i=1$, $i\in[n]$, we have 
\[
u_i^n=\arcsin\left(\frac{2i-n-1}{n-1}\xi\right)+V(t),\quad 
\xi=\frac{(n-1)\nu}{pKC_\D^{\sigma}+b_1},
\]
which coincides with \eqref{eqn:dsol}.
The synchronized solution \eqref{eqn:dsol0}
 suffers such bifurcations as detected in Theorem~\ref{thm:3b}
 and has the same stability type as the equilibrium $v^\sigma$
 in \eqref{eqn:dsys1}, which is determined in Theorem~\ref{thm:3c}.
In particular, it is stable for $\xi<\xi_*$ and unstable for $\xi>\xi_*$
 if $\sigma_i=1$ for $i\in[n]$,
 and always unstable if $\sigma_i=-1$ for some $i\in[n]$,
 under the hypotheses of Theorem~\ref{thm:3c}.

\section{Continuum Limits}
In Section~3, we obtained a complete description of equilibria
 in the CKM \eqref{eqn:dsys} with the natural frequencies \eqref{eqn:omegai}
 on complete graphs, including their bifurcation structure and stability.
In this section, based on the fundamental results in Section 2 and the characterization of equilibria and their stability,
 we study continuous and discontinuous solutions and their stability
 in the CL \eqref{eqn:csys} with the frequency function \eqref{eqn:omega}.
A similar approach was utilized in \cite{Y24a}.

The CL \eqref{eqn:csys} has two synchronized continuous solutions
 $u=U(x)+V(t)$ and $u=\pi-U(x)+V(t)$ in \eqref{eqn:csol}
 for which the function $U(x)$ is written as
\[
U(x)=\arcsin\left(\frac{a(x-\tfrac{1}{2})}{pKC+b_1}\right),
\]
where  the constant $C>0$ or $C<0$ satisfies
\begin{align}
C=&\frac{pKC+b_1}{a}\Biggl(\arcsin\left(\frac{a}{2(pKC+b_1)}\right)\notag\\
& +\frac{a}{2(pKC+b_1)}\sqrt{1-\left(\frac{a}{2(pKC+b_1)}\right)^2}\Biggr),
\label{eqn:C1}
\end{align}
which follows from \eqref{eqn:C}.

\begin{prop}
\label{prop:4a}
Let $pK/a<2/\pi$.
The two continuous solutions $u=U(x)+V(t)$ and $u=\pi-U(x)+V(t)$ in \eqref{eqn:csol}
 exist for the CL \eqref{eqn:csys} if
\begin{equation}
b_1\ge\tfrac{1}{2}a-\tfrac{1}{4}\pi pK\quad\mbox{and}\quad
b_1\ge\tfrac{1}{2}a+\tfrac{1}{4}\pi pK,
\label{eqn:prop4a}
\end{equation}
respectively.
\end{prop}

\begin{proof}
Let $pK/a<2/\pi$.
We only have to show that
 Eq.~\eqref{eqn:C1} has positive and negative solutions if Eq.~\eqref{eqn:prop4a} holds.
Letting
\[
\varphi(\eta)=\frac{\beta_0-\left(\arcsin\eta+\eta\sqrt{1-\eta^2}\right)}{\eta},\quad
\beta_0=\frac{a}{pK},
\]
we rewrite \eqref{eqn:C1} as
\[
b_1=\pm\tfrac{1}{2}pK\varphi\left(\frac{a}{2(pKC+b_1)}\right)
\]
where the upper and lower signs, respectively,
 correspond to the cases $C>0$ and $C<0$. 
In addition,
 $\varphi(\eta)$ is monotonically decreasing on $(-1,0)$ and $(0,1)$, and
\begin{align*}
&
\lim_{\eta\to+0}\varphi(\eta)=+\infty,\quad
\varphi(1)=\beta_0-\tfrac{1}{2}\pi,\\
&
\lim_{\eta\to-0}\varphi(\eta)=-\infty,\quad
\varphi(-1)=-\beta_0-\tfrac{1}{2}\pi.
\end{align*}
Hence, Eq.~\eqref{eqn:C1} has solutions $C>0$ and $C<0$, respectively, if
\[
b_1\ge\tfrac{1}{2}pK\left(\beta_0-\tfrac{1}{2}\pi\right)\quad\mbox{and}\quad
b_1\ge\tfrac{1}{2}pK\left(\beta_0+\tfrac{1}{2}\pi\right),
\]
which are equivalent to \eqref{eqn:prop4a}.
Thus, we obtain the desired result.
\end{proof}

\begin{rmk}
\label{rmk:4a}
As shown in Section~$7$ of {\rm\cite{Y24a}}, if $pK/a<2/\pi$,
 then such a continuous solution as $u=U(x)+V(t)$ and $u=\pi-U(x)+V(t)$
 does not exist in the CL \eqref{eqn:csys} with $b_1,b_0=0$,
 i.e., in the CL \eqref{eqn:csys0}.
\end{rmk}

Let $m_\pm$ be nonnegative integers that may be infinite,
 and let $\hat{I}_j^{\pm}\subset[0,1]$, $j\in[m_\pm]$, be intervals
 such that $\hat{I}_j^-\in[0,\tfrac{1}{2}]$, $\hat{I}_j^+\in[\tfrac{1}{2},1]$,
 $\hat{I}_j^-\cap \hat{I}_k^-,\hat{I}_j^+\cap \hat{I}_k^+=\emptyset$ and 
\[
\hat{I}=\bigcup_{j=1}^{m_-}\hat{I}_j^-\cup\bigcup_{j=1}^{m_+}\hat{I}_j^+\neq\emptyset,
\]
where the interiors of $\hat{I}_j^\pm$, $j\in[m_\pm]$, may be empty. 
The CL \eqref{eqn:csys} has discontinuous solutions
\begin{equation}
u(t,x)=\begin{cases}
U(x)+V(t)
 & \mbox{for $x\in[0,1]\setminus\hat{I}$;}\\
\pi-U(x)+V(t) & \mbox{for $x\in\hat{I}_j^+$, $j\in[m_+]$;}\\
-U(x)-\pi+V(t) & \mbox{for $x\in\hat{I}_j^-$, $j\in[m_-]$,}
\end{cases}
\label{eqn:csol3}
\end{equation}
where the constant $C$ in $U(x)$ satisfies
\[
C=\int_{[0,1]\setminus\hat{I}}\sqrt{1-\left(\frac{a(x-\tfrac{1}{2})}{pKC+b_1}\right)^2}\d x
 -\int_{\hat{I}}\sqrt{1-\left(\frac{a(x-\tfrac{1}{2})}{pKC+b_1}\right)^2}\d x.
\]
We see that the synchronized solution given by \eqref{eqn:dsyssol}
 converges to \eqref{eqn:csol3} as $n\to\infty$ when
\begin{equation}
\sigma_i=\begin{cases}
1 & \mbox{if $i/n\in[0,1]\setminus\hat{I}$;}\\
-1 & \mbox{if $i/n\in\hat{I}$.}
\end{cases}
\label{eqn:sigma1}
\end{equation}
Using Theorems~\ref{thm:2c} and \ref{thm:2d}, we prove the following

\begin{thm}
\label{thm:4a}
Let $pK/a<2/\pi$ and $b_1>0$.
Then the following hold$:$
\begin{enumerate}
\setlength{\leftskip}{-1.5em}
\item[(i)]
The continuous solution $u=U(x)+V(t)$ is asymptotically stable
 while the continuous solution $u=\pi-U(x)+V(t)$ is unstable
 when condition \eqref{eqn:prop4a} holds$;$
\item[(ii)]
The discontinuous solution \eqref{eqn:csol3} with $\hat{I}\neq\emptyset$ is unstable
 if it exists and is different from $u=U(x)+V(t)$ in the sense of $L^2(I)$.
\end{enumerate}
\end{thm}
 
\begin{proof}
The synchronized solution \eqref{eqn:dsol0}
 with $\sigma_i=1$ (resp. $\sigma_i=-1$), $i\in[n]$,
 which is asymptotically stable for $\xi<\xi_0$
 (resp. unstable for $\xi\in(0,1)$) in the CKM \eqref{eqn:dsys} by Theorem~\ref{thm:3c},
 converges to the continuous solution $u=U(x)+V(t)$
 (resp. $u=\pi-U(x)+V(t)$) as $n\to\infty$.
Note that by Theorem~\ref{thm:3b}(i)
 $\bar{\chi}^\sigma(\xi)$ has a unique extremum, which is a local maximum, at $\xi=\xi_0$.

Let $\sigma_i=1$, $i\in[n]$. 
Since
\[
\chi^\sigma(\xi)\to 2\int_0^1\sqrt{1-\xi^2x^2}\,\d x
=\frac{\xi\sqrt{1-\xi^2}+\arcsin\xi}{2\xi},
\]
we have
\[
\bar{\chi}^\sigma(\xi)=\frac{\xi}{\beta-\xi\chi^\sigma(\xi)}
\to\frac{2\xi}{\beta_0-(\xi\sqrt{1-\xi^2}+\arcsin\xi)}=:\bar{\chi}_0(\xi).
\]
Noting that $\bar{\chi}_0(\xi)=2/\varphi(\xi)$,
we see that $\bar{\chi}_0(\xi)$ is monotonically increasing on $(0,1)$.
This implies that $\xi_0\to 1$ as $n\to\infty$.
Part~(i)  follows from Theorem~\ref{thm:2c}(i).

We now prove part~(ii).
The synchronized solution \eqref{eqn:dsol0} with \eqref{eqn:sigma1},
 which is unstable by Theorem~\ref{thm:3c}(ii),
 converges to the discontinuous solution \eqref{eqn:csol3} in $L^2(I)$ as $n\to\infty$.
Hence, by Theorem~\ref{thm:2d},
 if it is different from $u=U(x)+V(t)$ in the sense of $L^2(I)$,
 then the discontinuous solution \eqref{eqn:csol3} with $\hat{I}\neq\emptyset$
 is unstable.
\end{proof}

\begin{rmk}
As seen from the proof of Theorem~$\ref{thm:4a}$,
 condition~$\eqref{eqn:prop3a}$ becomes \eqref{eqn:prop4a} as $n\to\infty$.
\end{rmk}

From Theorems~\ref{thm:3b} and \ref{thm:3c}
 we obtain the following results for the CKM \eqref{eqn:dsys} on complete graphs:
\begin{itemize}
\setlength{\leftskip}{-2em}
\item
The synchronized solution \eqref{eqn:dsyssol} for $v^\sigma$ with $\sigma_i=1$, $i\in[n]$,
 undergoes a saddle-node bifurcation at $\xi=\xi_0$, i.e., $b_1=pK/\bar{\chi}^\sigma(\xi_0)$,
 and it changes from stable to unstable there
\item
In addition, it undergoes a pitchfork bifurcation at $\xi=1$, i.e., $b_1=pK/\bar{\chi}^\sigma(1)$,
 where it changes to a synchronized solution
 with $\sigma_1=\sigma_n=-1$ and $\sigma_j=1$, $j\in[n-1]\setminus\{1\}$,
 while the two synchronized solutions
 with $\sigma_1=-1$, $\sigma_{n}=1$ or $\sigma_1=1$, $\sigma_{n}=-1$
 and $\sigma_j=1$, $j\in[n-1]\setminus\{1\}$, are born.
\end{itemize}
As shown in the proof of Theorem~\ref{thm:4a},
 we have $\xi_0\to1$ as $n\to\infty$. 
Hence, in this limit, the saddle-node and pitchfork bifurcations coalesce.
By Theorem~\ref{thm:4a},
 these four solutions converge to a single asymptotically stable solution in the sense of $L^2(I)$
 in the CL \eqref{eqn:csys}.
More generally,
if $\sigma_i=1$ for all but a fixed finite positive number of indices $i\in[n]$,
 then by Theorem~\ref{thm:4a}(ii)
 the synchronized solution \eqref{eqn:dsyssol} for $v^\sigma$ is unstable,
 but it converges in $L^2(I)$
 to the asymptotically stable solution $u=U(x)+V(t)$ as $n\to\infty$.
Thus, the bifurcation behavior of the CL \eqref{eqn:csys}
 is more subtle than that of finite-dimensional dynamical systems
 such as the CKM \eqref{eqn:dsys}.
Such behavior was previously detected
 for the KM \eqref{eqn:dsys0} on complete simple graphs in \cite{Y24a}.

\section{Controlled Kuramoto Model on Uniform Random Graphs}

We turn to the CKM \eqref{eqn:dsys} with the natural frequencies \eqref{eqn:omegai}
 on uniform random dense and sparse graphs
 given by \eqref{eqn:rdg} and \eqref{eqn:rsg}, respectively.
In contrast to the case of complete graphs,
 the randomness of the coupling makes a direct analytical characterization of equilibria and their stability
 in the CKM \eqref{eqn:dsys} highly difficult.
However, the relationships between the CKM \eqref{eqn:dsys} and CL \eqref{eqn:csys} 
 stated in Section~2 allow us to describe fundamental dynamical behavior
 in the former, by using the results of Section~4 for the latter.

Let $\bar{\mathbf{u}}(t)$ (resp. $\hat{\mathbf{u}}(t)$) denote solutions $u=U(x)+V(t)$
 (resp. $u=\pi-U(x)+V(t)$ or \eqref{eqn:csol3}) to the CL \eqref{eqn:csys},
 such that $\bar{\mathbf{u}}(t)$ and $\hat{\mathbf{u}}(t)$ are different in $L^2(I)$.
Using Corollary~\ref{cor:2a} and Theorems~\ref{thm:2e} and \ref{thm:4a},
 we prove the following.

\begin{thm}\
\label{thm:5a}
\begin{enumerate}
\setlength{\leftskip}{-1.5em}
\item[(i)]
For any $\epsilon,\tau>0$ there exists $\delta,N>0$ such that for any $n>N$
 if $\mathbf{u}_n(t)$ is any solution to the CKM \eqref{eqn:dsys} satisfying
\[
\|\mathbf{u}_n(0)-\bar{\mathbf{u}}(0)\|<\delta,
\]
then
\[
\max_{t\in[0,\tau]}\|\mathbf{u}_n(t)-\bar{\mathbf{u}}(t)\|<\epsilon
\quad\mbox{a.s.}
\]
Moreover,
\[
\lim_{t\to\infty}\lim_{n\to\infty}\|\mathbf{u}_n(t)-\bar{\mathbf{u}}(t)\|
 =0\quad\mbox{a.s.}
\]
\item[\rm(ii)]
For any $\epsilon,\delta>0$
 there exist $\tau,N>0$ such that for $n>N$
a solution $\mathbf{u}_n(t)$ to the CKM \eqref{eqn:dsys} can be chosen so that
\[
\|\mathbf{u}_n(0)-\hat{\mathbf{u}}(0)\|<\delta
\]
and
\[
\|\mathbf{u}_n(\tau)-\hat{\mathbf{u}}(\tau)\|>\epsilon\quad\mbox{a.s.}
\]
\end{enumerate}
\end{thm}

\begin{proof}
By Theorem~\ref{thm:4a},
 $\bar{\mathbf{u}}(t)$ and $\hat{\mathbf{u}}(t)$ are, respectively,
 asymptotically stable and unstable solutions in the CL \eqref{eqn:csys}.
Hence, parts~(i) and (ii) immediately  follow
 from Corollary~\ref{cor:2a} and Theorem~\ref{thm:2e}, respectively.
\end{proof}

\begin{rmk}
For Kuramoto-type models on complete simple or general graphs,
 mean-field limits have been extensively studied,
 leading to kinetic equations for the probability density of oscillator phases
{\rm \cite{ABVRS05,C15,CM19a,CM19b}}.
These approaches describe the evolution of phase distributions
 rather than the time histories of individual phases,
 which are random processes when the natural frequencies
 or the underlying graph $G_n$ is random.
Moreover, in many concrete mean-field analyses such as {\rm\cite{C15,CM19a,CM19b}},
 the natural frequencies are modeled as random variables
 with a prescribed probability distribution,
 typically satisfying a unimodality assumption.
 
In contrast, we consider deterministic, uniformly spaced natural frequencies,
 for which synchronized solutions of the CKM \eqref{eqn:dsys}
 and their bifurcations can be analyzed directly.
The CL \eqref{eqn:csys} adopted here
 describes the almost sure convergence of solutions of the CKM \eqref{eqn:dsys}
 to those of \eqref{eqn:csys} as $n\to\infty$.
From the mean-field perspective,
 such limiting states correspond to singular measures,
 typically delta distributions,
 and capturing their stability and bifurcation structure
 within the kinetic framework is highly nontrivial.
\end{rmk}

Thus, the solution $u=U(x)+V(t)$ behaves 
 (resp. the solutions $u=\pi-U(x)+V(t)$ and \eqref{eqn:csol3} behave)
 as if it is an asymptotically stable one (resp. they are unstable ones)
 in the CKM \eqref{eqn:dsys} on uniform random dense and sparse graphs.
Similar behavior was previously numerically observed or theoretically detected
 for the KMs defined on some random dense and sparse graphs and their CLs 
 including the KM \eqref{eqn:dsys0} and CL \eqref{eqn:csys0}
 when the natural frequencies are deterministic or random,
 in \cite{IY23,Y24b,Y24c,Y24d,Y24e}.

\section{Numerical Simulations}

\begin{figure}
\centering
\includegraphics[width=0.48\textwidth]{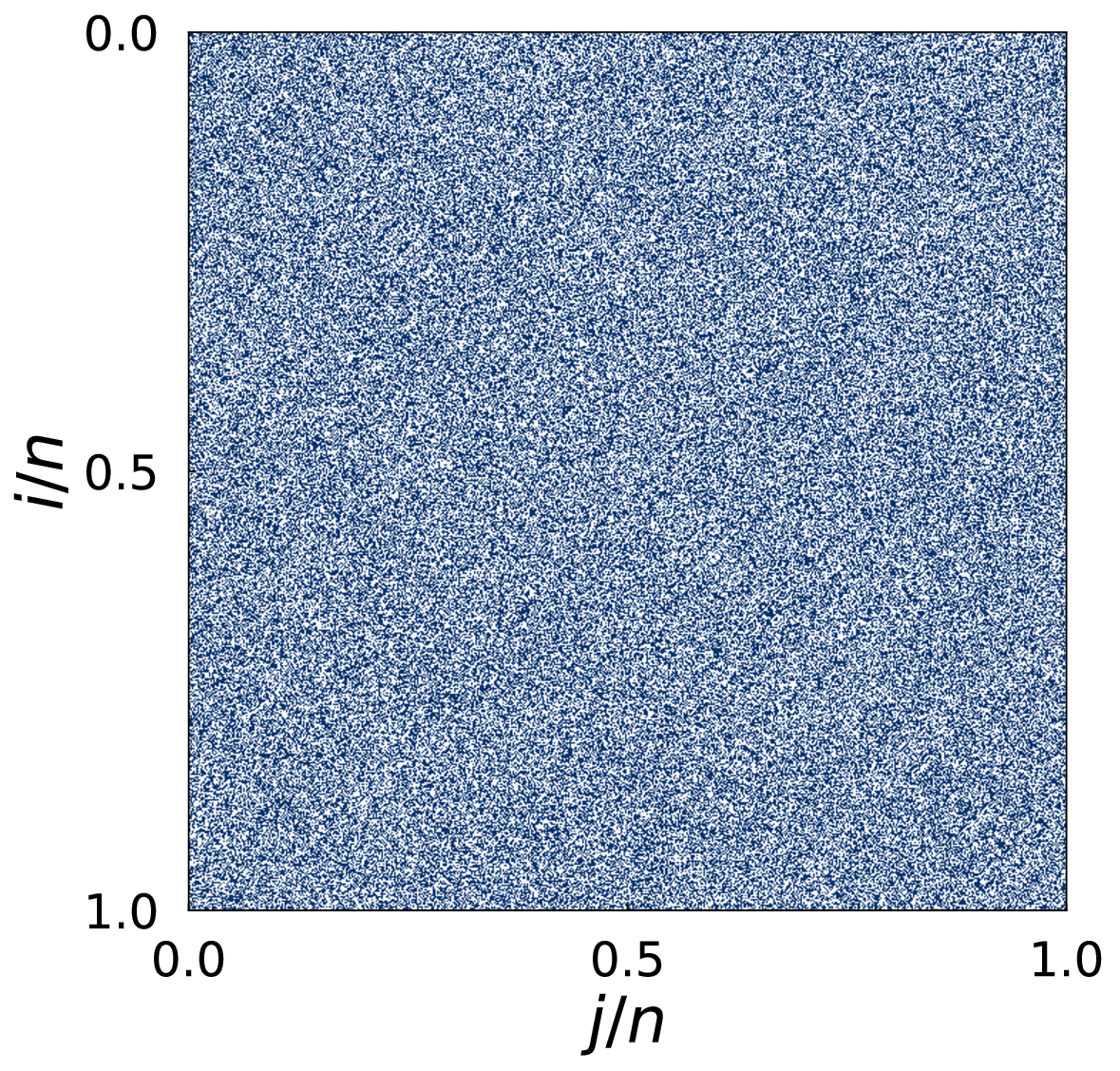}
\includegraphics[width=0.48\textwidth]{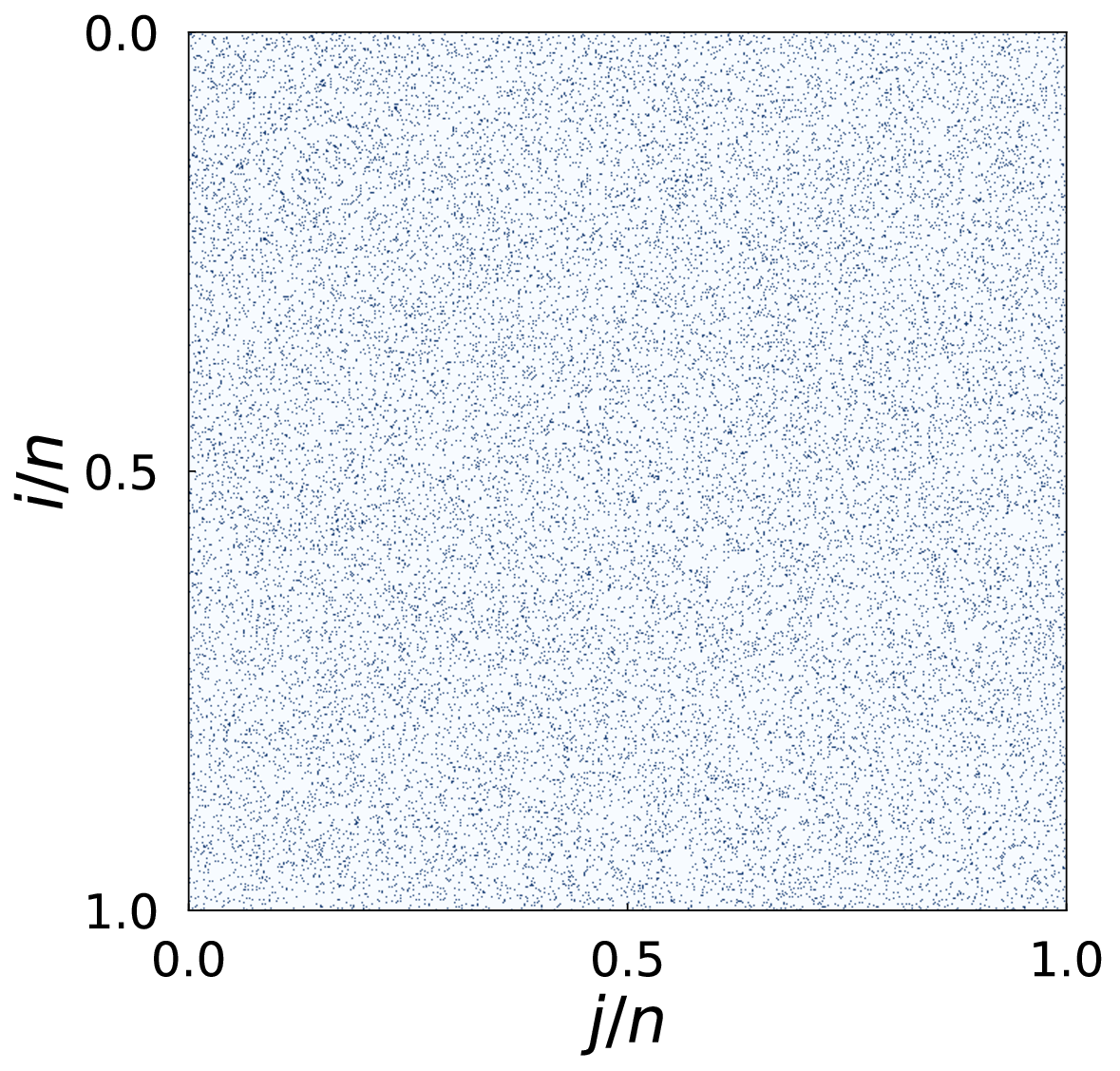}
\caption{
Pixel pictures of sampled weighted matrices 
 for the random undirected graphs given by $w_{ij}^n=1$, $i,j\in[n]$ 
 with probability \eqref{eqn:Pdg} and \eqref{eqn:Psg} for $n=1000$:
(a) Dense graph with $p=0.5$;
(b) Sparse graph with $p=0.5$ and $\gamma=0.3$.
Blue pixels correspond to entries with $w_{ij}=1$,
whereas light blue pixels correspond to the remaining entries.
\label{fig:5a}}
\end{figure}

We finally present numerical simulation results for the CKM \eqref{eqn:dsys}.
We treat the following three cases for the graph $G_n$:
\begin{itemize}
\setlength{\leftskip}{-1.3em}
\item[(i)]
Complete simple graph, which is deterministic dense with $p=1$;
\item[(ii)]
Random undirected dense graph in which $w_{ij}^n=1$ occurs with probability
\begin{equation}
\Pset (i\sim j) =p,\quad i,j\in[n];
\label{eqn:Pdg}
\end{equation}
\item[(iii)]
Random undirected sparse graph in which $w_{ij}^n=1$ occurs  with probability
\begin{equation}
\Pset (i\sim j) =n^{-\gamma}p,\quad i,j\in[n],
\label{eqn:Psg}
\end{equation}
where $\gamma\in(0,\tfrac{1}{2})$.
Recall that $\alpha_n=n^{-\gamma}$
 and note that $\alpha_n^{-1}=n^\gamma>1$ for any $n>1$.
\end{itemize}
We choose $p=0.5$ for case~(ii), and $p=0.5$, $\gamma=0.3$ for case~(iii).
In Figs.~\ref{fig:5a}(a) and (b), we display numerical samples
 of the weight matrices for the random undirected dense and sparse graphs
 in cases~(ii) and (iii), respectively,
 for $n=1000$, $p=0.5$ and $\gamma=0.3$.

We performed numerical simulations for the CKM \eqref{eqn:dsys}
 on the three types of graphs with $n=1000$, $V_1=b_0=1$ and $V_0=1$
 using the DOP853 solver \cite{HNW93}.
We also chose $K=0.5$ and $a=1$ for case~(i),
 and $K=0.5$ and $a=0.5$ for cases~(ii) and (iii).
The continuous solution $u=U(x)+V(t)$
 does not exist in the CL \eqref{eqn:csys} with $b_1,b_0=0$, in cases~(i)-(iii)
 since $pK/a=0.5<2/\pi=0.63661\ldots$ (see Remark~\ref{rmk:4a}).
The initial values $u_i^n(0)$, $i\in[n]$ were independently randomly chosen
according to the uniform distribution on $[-\pi, \pi]$.

\begin{figure}[t]
\centering
\includegraphics[width=0.47\textwidth]{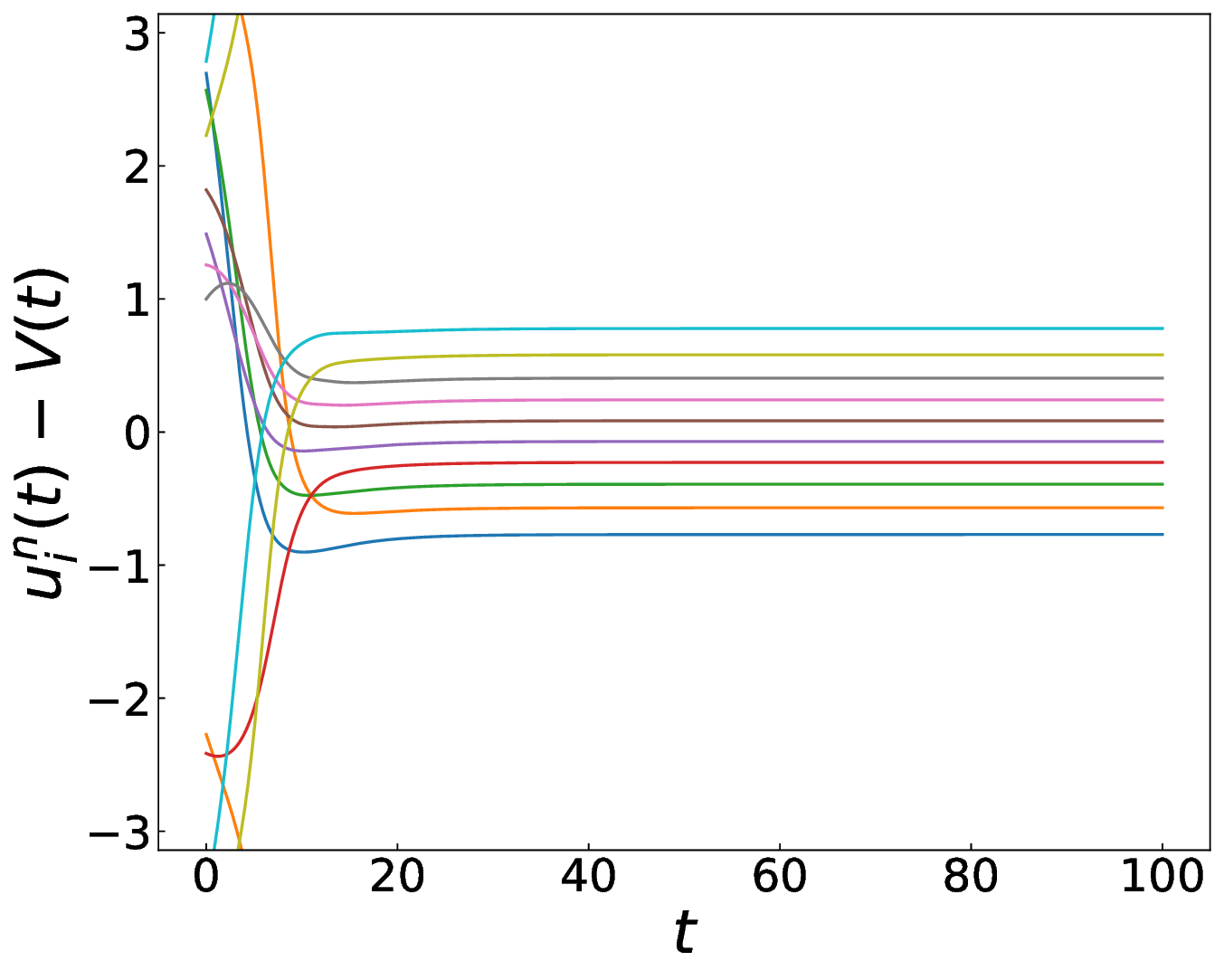}\\[1ex]
\includegraphics[width=0.47\textwidth]{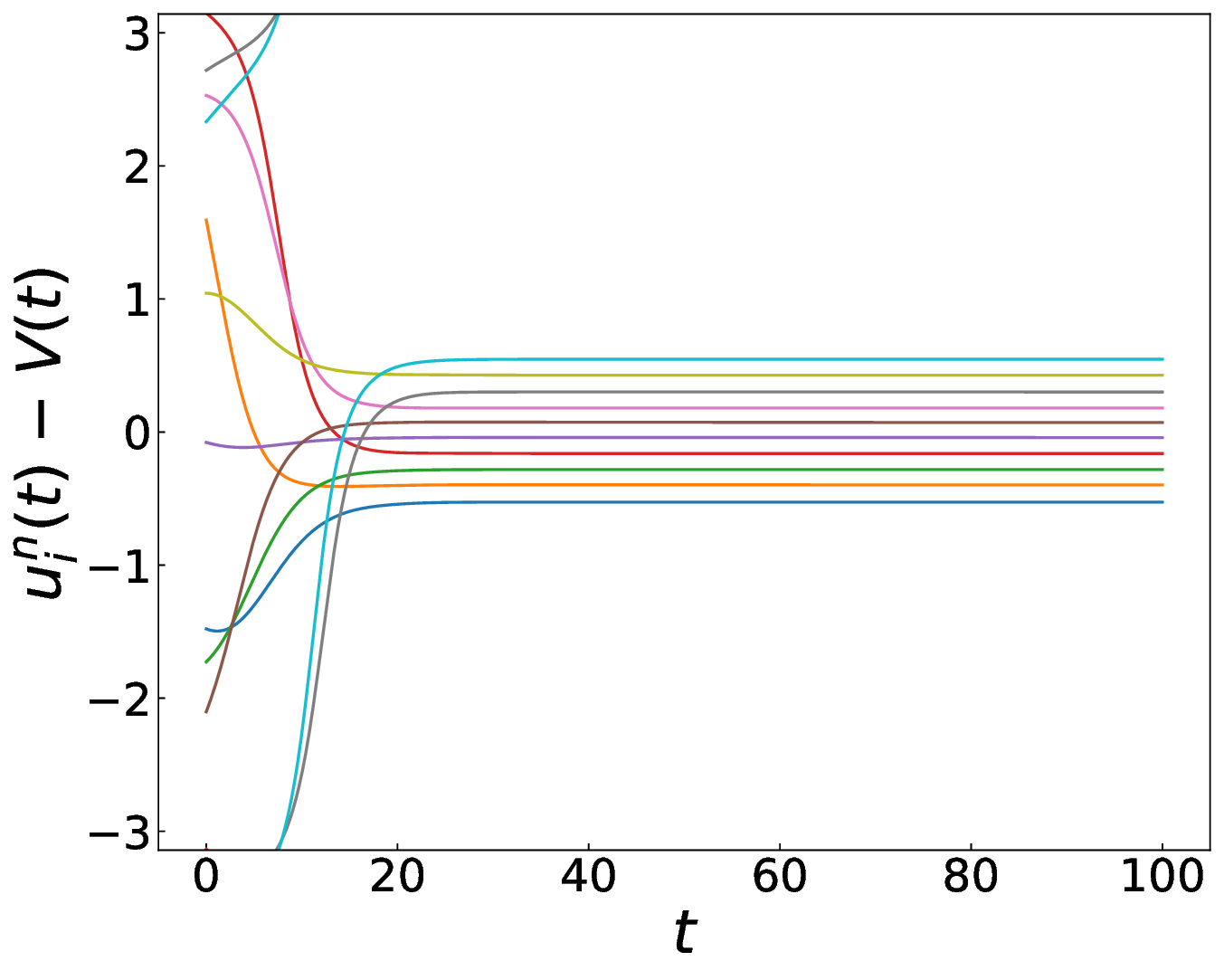}\qquad
\includegraphics[width=0.47\textwidth]{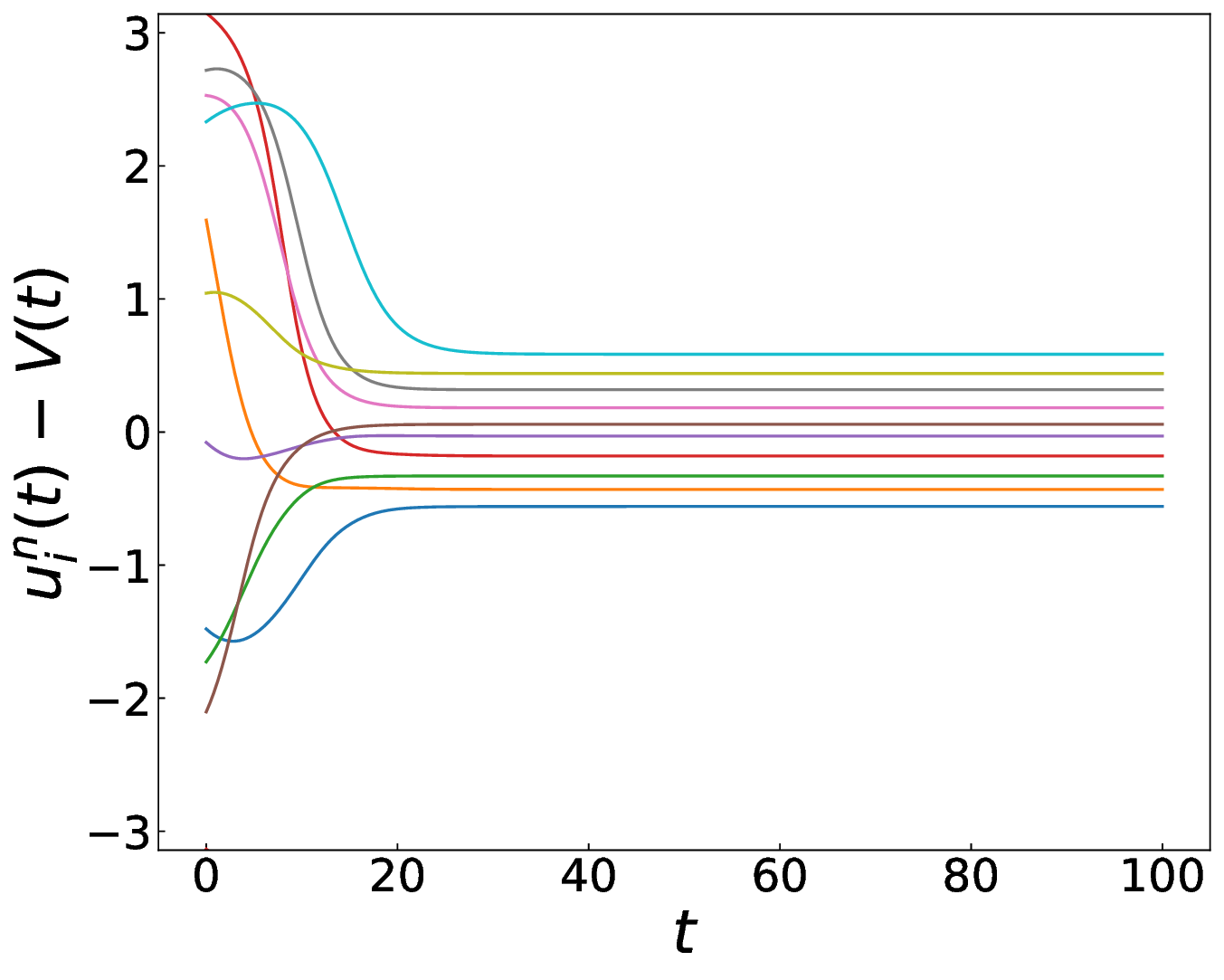}
\caption{
Numerical simulation results of the CKM \eqref{eqn:dsys}
 with $n=1000$, $K=0.5$, $V_1,b_0=1$, $V_0=1$ and $ b_1=0.2$:
(a) $(a,p)=(1,1)$ in case~(i);
(b) $(0.5,0.5)$ in case~(ii);
(c) $(a,p,\gamma)=(0.5,0.5,0.3)$ in case~(iii).
The time histories of every 100th node, from the 50th to the 950th node,
are shown in different colors.}
\label{fig:5b}
\end{figure}

In Figs.~\ref{fig:5b}(a), (b) and (c),
 we give the time-histories of  every 100th node (from 50th to 950th) in cases~(i), (ii) and (iii), respectively,
 for $b_1=0.2$, in which the asymptotically stable solution $u=U(x)+V(t)$ exists in the CL \eqref{eqn:csys}.
The ordinates represent the deviations of the responses from the desired motion,
 $u_i^n(t)-V(t)$.
We observe that the responses rapidly converge to the synchronized solution
 around the desired motion, as predicted theoretically in Sections~3 and 5.

\begin{figure}
\centering
\includegraphics[width=0.47\textwidth]{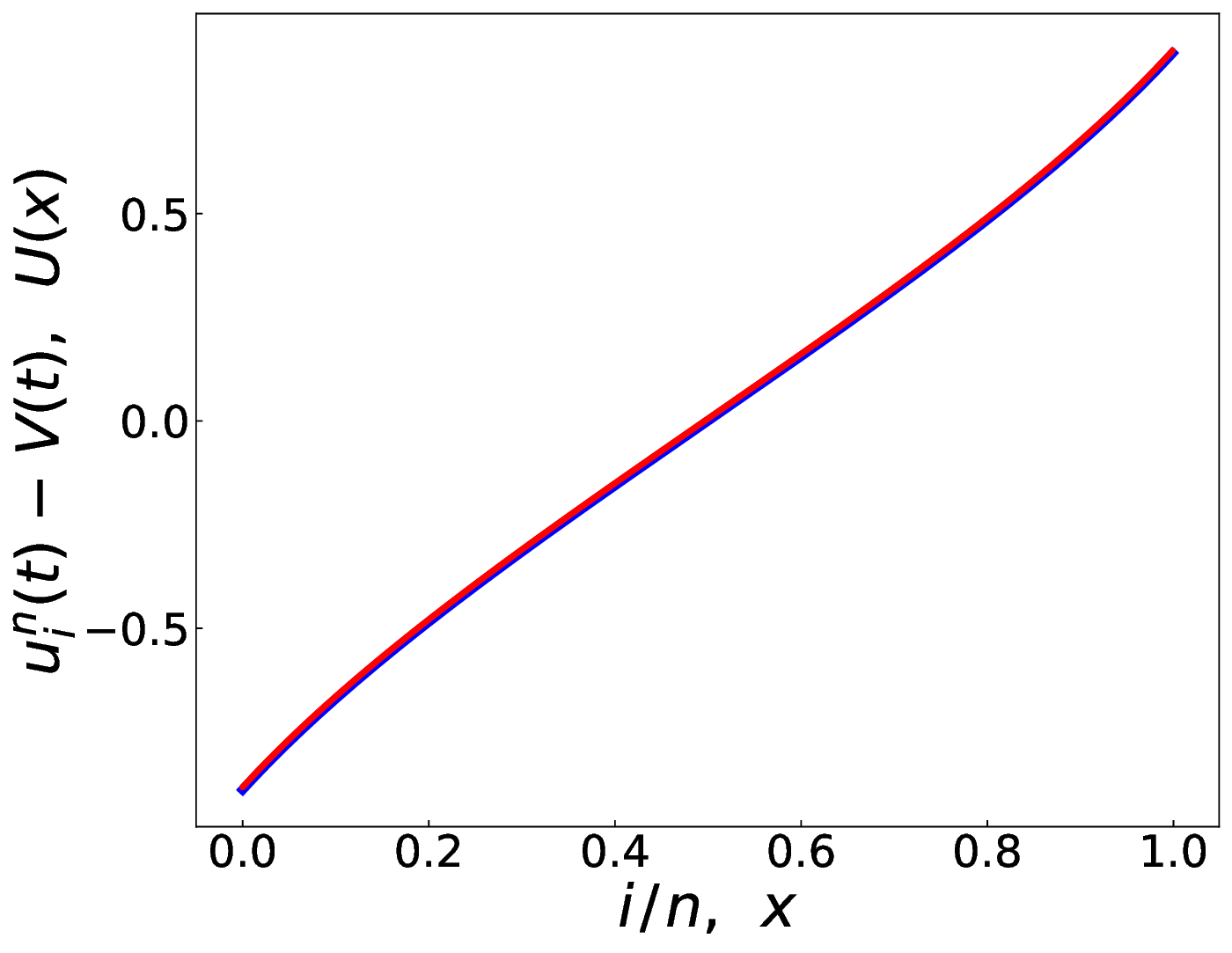}\\[1ex]
\includegraphics[width=0.47\textwidth]{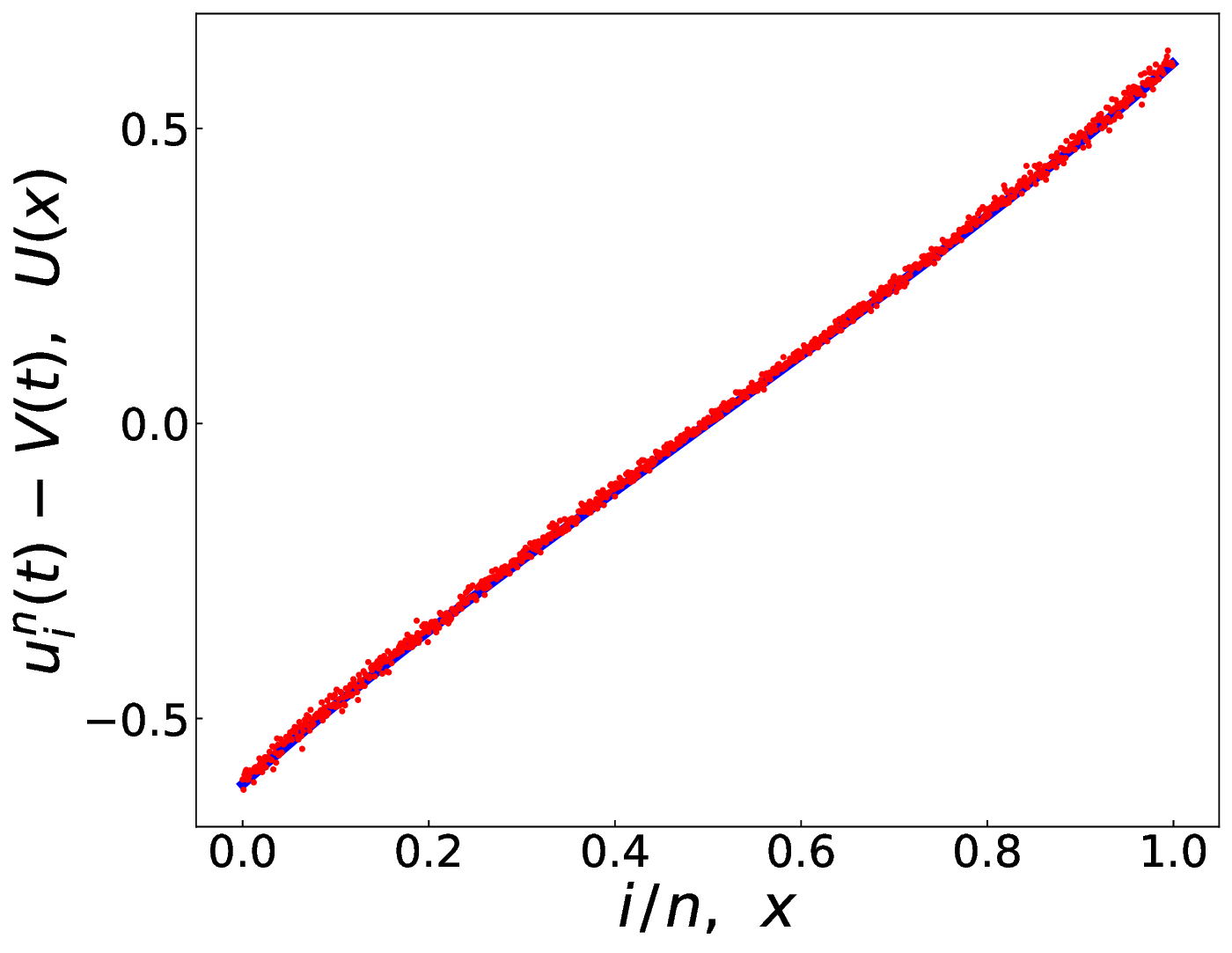}\qquad
\includegraphics[width=0.47\textwidth]{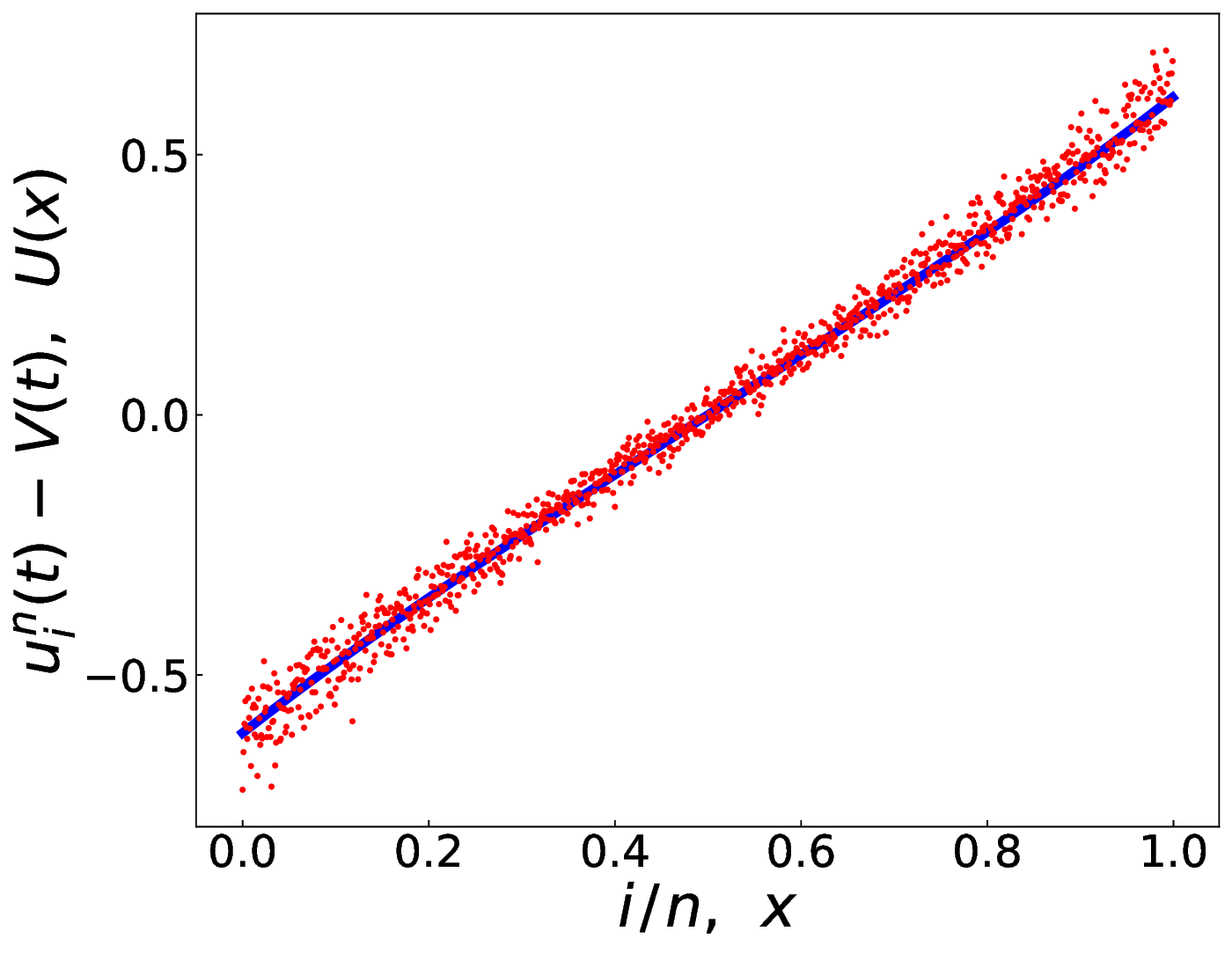}
\caption{
Deviations of steady-state responses from the desired motion in the CKM \eqref{eqn:dsys}
 with $n=1000$, $K=0.5$, $V_1,b_0=1$, $V_0=1$ and $ b_1=0.2$:
(a) $(a,p)=(1,1)$ in case~(i);
(b) $(0.5,0.5)$ in case~(ii);
(c) $(a,p,\gamma)=(0.5,0.5,0.3)$ in case~(iii).
Here $u_i^n(t)-V(t)$, $i\in[n]$, with $t=100$ are plotted as red dots.
The blue line represents the corresponding theoretical predictions
 computed from the synchronized solution $u=U(x)+V(t)$
 in the CL \eqref{eqn:csys}.}
\label{fig:5c}
\end{figure}

Figures~\ref{fig:5c}(a), (b) and (c) plot the deviations of the responses
 from the desired motion at $t=100$ for cases~(i), (ii) and (iii), respectively.
These states are regarded as steady states in view of Fig.~\ref{fig:5b}.
The initial condition and $b_1$ are chosen as in Fig.~\ref{fig:5b}.
The blue line in each panel gives the corresponding theoretical prediction
 computed from the synchronized solution $u=U(x)+V(t)$ in the CL \eqref{eqn:csys}.
In Figs.~\ref{fig:5c}(a) and (b), corresponding to cases~(i) and (ii),
 the numerical results are in excellent agreement with the theoretical predictions.
In Fig.~\ref{fig:5c}(c), corresponding to case~(iii),
 the agreement remains good, 
 although random fluctuations are visible.
Thus, we observe that the synchronized solution $u=U(x)+V(t)$
 behaves as it is an asymptotically stable one in the CKM \eqref{eqn:dsys}. 

\begin{figure}
\begin{center}
\includegraphics[width=0.47\textwidth]{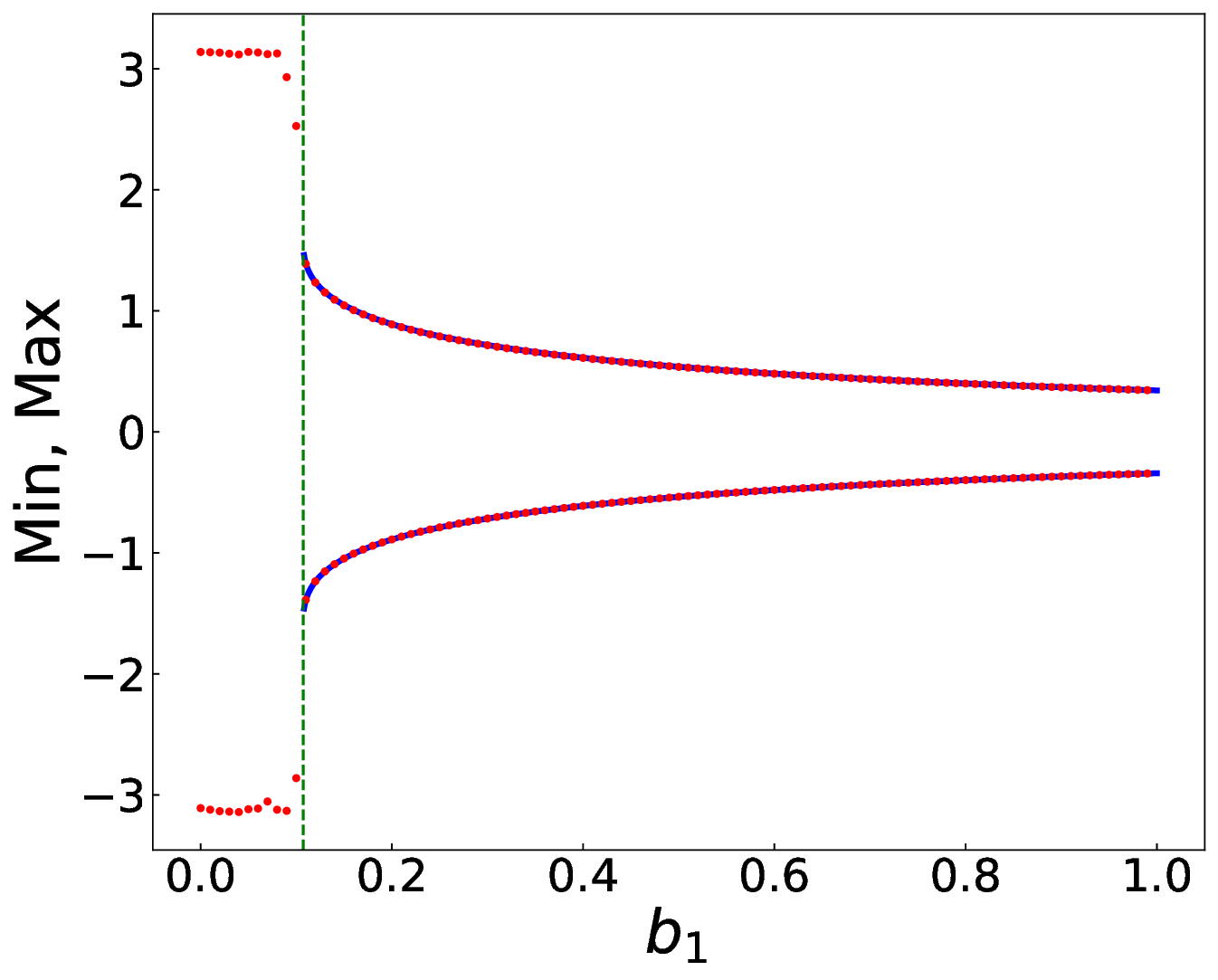}\\
\includegraphics[width=0.47\textwidth]{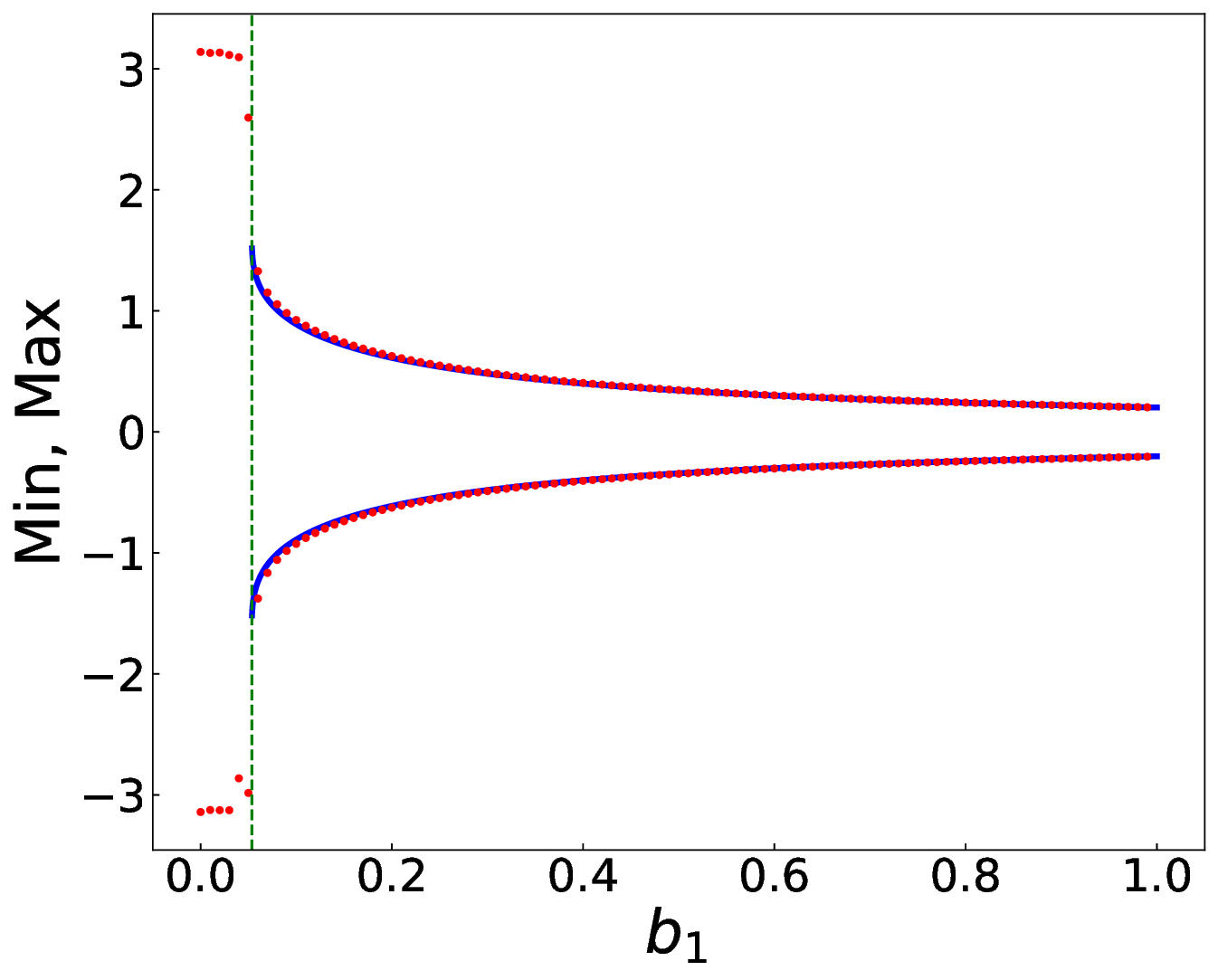}\qquad
\includegraphics[width=0.47\textwidth]{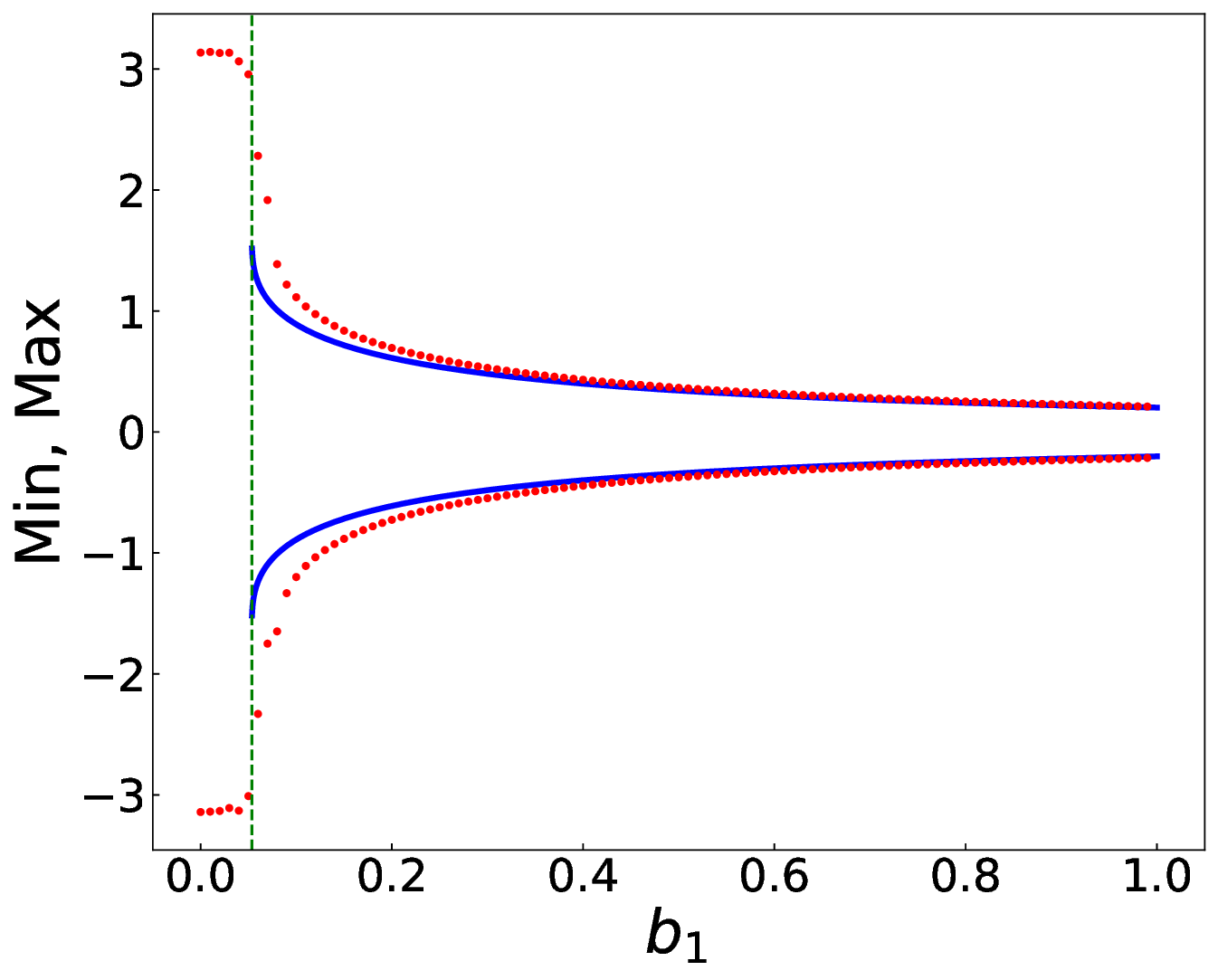}
\end{center}
\caption{
Maximal and minimal deviations of the steady-state responses
 from the desired motion in the CKM \eqref{eqn:dsys}
 with $n=1000$, $K=0.5$, $V_1,b_0=1$ and $V_0=1$
 when the feedback gain $b_1$ is changed:
(a) $(a,p)=(1,1)$ in case~(i);
(b) $(0.5,0.5)$ in case~(ii);
(c) $(a,p,\gamma)=(0.5,0.5,0.3)$ in case~(iii).
Here $\max_{i\in[n]}(u_i^n(t)-V(t))$
 and $\min_{i\in[n]}(u_i^n(t)-V(t))$, $i\in[n]$, for $t>0$ sufficiently large
 are plotted as red dots.
The blue line represents the corresponding theoretical predictions $\pm\Delta u$
 computed from the synchronized solution $u=U(x)+V(t)$
 in the CL \eqref{eqn:csys}, where $\Delta u$ is given by \eqref{eqn:du}.}
\label{fig:5d}
\end{figure}

Figures~\ref{fig:5d}(a), (b) and (c) display
 the maximal and minimal deviations of the steady-state responses
 from the desired motion in cases~(i), (ii) and (iii), respectively,
 when the feedback gain $b_1$ is varied.
The blue line in each figure represents the corresponding theoretical predictions
 $\pm\Delta u$ with
\begin{equation}
\Delta u=\arcsin\left(\frac{a}{2(pKC+b_1)}\right)
\label{eqn:du}
\end{equation}
computed from the synchronized solution $u=U(x)+V(t)$ in the CL \eqref{eqn:csys},
 where $C$ satisfies \eqref{eqn:C1},
 and the green line represents the critical value of $b_1$ given by \eqref{eqn:prop4a}
 at which the solution suddenly appears when $b_1$ is increased.
We observe that
 the steady-state responses tend to the desired motion as $b_1\to\infty$.
In Figs.~\ref{fig:5d}(a) and (b), corresponding to cases~(i) and (ii),
 the numerical results agree very well with the theoretical predictions.
In Fig.~\ref{fig:5d}(c), corresponding to case~(iii)
 the agreement remains good.

\section*{Acknowledgments}
The author thanks Donggeon Kim for his assistance in numerical computations.
This work was partially supported by JSPS KAKENHI Grant Number JP23K22409.



\section*{Data Availability}
Data sets generated during the current study are available from the author on reasonable request.

\appendix
 
\renewcommand{\theequation}{A.\arabic{equation}} 
\setcounter{equation}{0} 

\section{Derivation of \eqref{eqn:csol} and \eqref{eqn:dsol}}

We first derive the solution \eqref{eqn:csol} to the CL \eqref{eqn:csys}.
Substituting $u(t,x)=\tilde{U}(x)+V(t)$ with $V(t)=V_1t+V_0$ into \eqref{eqn:csys}, we have
\begin{align}
V_1=&\omega(x)
+pK\cos\tilde{U}(x)\int_I\sin\tilde{U}(y)\d y\notag\\
&-pK\sin\tilde{U}(x)\int _I\cos\tilde{U}(y)\d y
-b_1\sin\tilde{U}(x)+b_0.
\label{eqn:a1}
\end{align}
Integrating \eqref{eqn:a1} with respect to $x$ on $I$ and using \eqref{eqn:conb0},
 we obtain
\[
\int_I\sin\tilde{U}(x)\d x=0,
\]
so that Eq.~\eqref{eqn:a1} becomes
\[
V_1=\omega(x)-\tilde{C}\sin\tilde{U}(x)+b_0,
\]
where
\begin{align}
\tilde{C}=pK\int_I\cos\tilde{U}(x)\d x+b_1.
\label{eqn:acon1}
\end{align}
Obviously, $\tilde{C}\neq 0$ since $\omega(x)-V_1+b_0\equiv 0$ otherwise.
Hence, we have
\begin{equation}
\tilde{U}(x)=\arcsin\left(\frac{\omega(x)-V_1+b_0}{\tilde{C}}\right)
\label{eqn:tU1a}
\end{equation}
and
\begin{equation}
\tilde{U}(x)=\pi-\arcsin\left(\frac{\omega(x)-V_1+b_0}{\tilde{C}}\right),
\label{eqn:tU1b}
\end{equation}
where by \eqref{eqn:acon1}
\begin{align}
\tilde{C}=pK\int_I\sqrt{1-\left(\frac{\omega(x)-V_1+b_0}{\tilde{C}}\right)^2}\d x+b_1.
\label{eqn:acon1a}
\end{align}
and
\begin{align}
\tilde{C}=-pK\int_I\sqrt{1-\left(\frac{\omega(x)-V_1+b_0}{\tilde{C}}\right)^2}\d x+b_1.
\label{eqn:acon1b}
\end{align}
respectively.
Letting $\tilde{C}=pKC+b_1$ in \eqref{eqn:tU1a} and \eqref{eqn:acon1a},
 we obtain the solution \eqref{eqn:csol} and \eqref{eqn:C}.

We next derive the solution \eqref{eqn:dsol} to the CKM \eqref{eqn:dsys}
 with $w_{ij}^n=p$, $i,j\in[n]$, and $\alpha_n=1$.
We substitute $u_i^n(t)=\tilde{U}_i^n+V(t)$ into \eqref{eqn:dsys} to obtain
\begin{align}
V_1=&\omega_i^n
 +\frac{pK}{n}\cos\tilde{U}_i^n\sum_{j=1}^{n}\sin\tilde{U}_j^n
 -\frac{pK}{n}\sin\tilde{U}_i^n\sum_{j=1}^{n}\cos\tilde{U}_j^n
 -b_1\sin\tilde{U}_i^n+b_0.
\label{eqn:a2}
\end{align}
Summing \eqref{eqn:a2} from $i=1$ to $n$ and using \eqref{eqn:b0}, we obtain
\[
\sum_{i=1}^n\sin\tilde{U}_i^n=0,
\]
so that Eq.~\eqref{eqn:a2} becomes
\[
V_1=\omega_i^n
 -\tilde{C}_\D\sin\tilde{U}_i^n+b_0,
\]
where
\begin{equation}
\tilde{C}_\D=\frac{pK}{n}\sum_{i=1}^n\cos\tilde{U}_i^n+b_1.
\label{eqn:acon2}
\end{equation}
Obviously, $\tilde{C}_\D\neq 0$ since $\omega_i^n-V_1+b_0= 0$ for any $i\in[n]$ otherwise.
Hence we have
\begin{equation}
\tilde{U}_i^n=\arcsin\left(\frac{\omega_i^n-V_1+b_0}{\tilde{C}_\D}\right),\quad
i\in[n],
\label{eqn:tU2a}
\end{equation}
and
\begin{equation}
\tilde{U}_i^n=\pi-\arcsin\left(\frac{\omega_i^n-V_1+b_0}{\tilde{C}_\D}\right),\quad
i\in[n],
\label{eqn:tU2b}
\end{equation}
where by \eqref{eqn:acon2}
\begin{equation}
\tilde{C}_\D=\frac{pK}{n}\sum_{j=1}^n
 \sqrt{1-\left(\frac{\omega_j^n-V_1+b_0}{\tilde{C}_\D}\right)^2}+b_1
\label{eqn:acon2a}
\end{equation}
and
\begin{equation}
\tilde{C}_\D=-\frac{pK}{n}\sum_{j=1}^n
 \sqrt{1-\left(\frac{\omega_j^n-V_1+b_0}{\tilde{C}_\D}\right)^2}+b_1,
\label{eqn:acon2b}
\end{equation}
respectively.
Letting $\tilde{C}_\D=pKC_\D+b_1$ in \eqref{eqn:tU2a} and \eqref{eqn:acon2a},
 we obtain the solution \eqref{eqn:dsol} and \eqref{eqn:CD0}.


\end{document}